\documentclass[reqno,a4paper]{amsart}
\usepackage{eucal,amsfonts,amssymb,amsmath,amsthm,epsfig,mathrsfs}
\usepackage{xcolor}
\usepackage{refcheck}

\usepackage{amscd,amsxtra}
\usepackage{enumerate}
\usepackage{latexsym}
\usepackage{dsfont}
\usepackage{amsfonts}
\usepackage{amsmath}
\usepackage{amssymb}
\usepackage{amsthm}
\usepackage{graphicx}
\usepackage{mathrsfs}
\allowdisplaybreaks

 \makeatletter \@addtoreset{equation}{section}

\makeatother \makeatletter

\newtheorem{theorem}{Theorem}[section]
\newtheorem{corollary}[theorem]{Corollary}
\newtheorem{lemma}[theorem]{Lemma}
\newtheorem{proposition}[theorem]{Proposition}
\theoremstyle{definition}

\newtheorem{remark}[theorem]{Remark}
\newtheorem*{notation}{Notation}
\newtheorem{hyp}[theorem]{Hypotheses}

\newcommand{\C}{{\mathbb C}}
\newcommand{\R}{{\mathbb R}}
\newcommand{\N}{{\mathbb N}}

\newcommand{\bd}{\begin{defi}}
\newcommand{\ed}{\end{defi}}

\newcommand{\be}{\begin{equation}}
\newcommand{\ee}{\end{equation}}
\newcommand{\barr}{\begin{array}}
\newcommand{\earr}{\end{array}}
\newcommand{\bmn}{\begin{eqnarray}}
\newcommand{\emn}{\end{eqnarray}}
\newcommand{\bnm}{\begin{eqnarray*}}
\newcommand{\enm}{\end{eqnarray*}}
\newcommand{\bln}{\begin{subequations}}
\newcommand{\eln}{\end{subequations}}

\newcommand{\ba}{\begin{align}}
\newcommand{\ea}{\end{align}}
\newcommand{\banm}{\begin{align*}}
\newcommand{\eanm}{\end{align*}}

\title[Generation results and heat kernel estimates]{On Schr\"odinger type operators with unbounded coefficients: Generation and heat kernel estimates}
\author{Luca Lorenzi}
\address{Dipartimento di Matematica, Universit\`a degli Studi di Parma, Parco Area delle Scienze 53/A, I-43124 PARMA, Italy.}
\email{luca.lorenzi@unipr.it}
\author{Abdelaziz Rhandi}
\address{Dipartimento di Matematica, Universit\`a di Salerno, Via Ponte Don Melillo, 84084 FISCIANO (Sa), Italy.}
\email{arhandi@unisa.it}
\keywords{Schr\"odinger type operators, sectorial operators, generation results, heat kernel}
\subjclass[2000]{47D07; 35J10, 35K05, 35K10}

\begin{document}

\begin{abstract}
We consider the Schr\"odinger type operator ${\mathcal A}=(1+|x|^{\alpha})\Delta-|x|^{\beta}$, for $\alpha\in [0,2]$ and $\beta\ge 0$. We prove that, for any $p\in (1,\infty)$, the minimal realization of operator ${\mathcal A}$ in $L^p(\R^N)$ generates a strongly continuous analytic semigroup $(T_p(t))_{t\ge 0}$.

For $\alpha\in [0,2)$ and $\beta\ge 2$, we then prove some upper estimates for the heat kernel $k$ associated to the semigroup
$(T_p(t))_{t\ge 0}$.
As a consequence we obtain an estimate for large $|x|$ of the eigenfunctions of ${\mathcal A}$. Finally, we extend such estimates to a class of divergence type elliptic operators.
\end{abstract}

\maketitle

\section{Introduction}
For any $\alpha,\beta\ge 0$ with $\alpha^2+\beta^2\neq 0$, let ${\mathcal A}$ be the
elliptic operator defined by
\begin{equation}
\mathcal{A}\varphi(x)=a(x)\Delta \varphi(x)-V(x) \varphi(x),\qquad\;\, x\in \R^N,
\label{operatore-A}
\end{equation}
on smooth functions $\varphi$, where $a(x)=1+|x|^\alpha$ and $V(x)=|x|^\beta$.

In the case when $\beta=0$ and $\alpha>0$,
generation results of analytic semigroups for suitable realizations $A_p$ of the operator ${\mathcal A}$ in $L^p(\R^N)$ have been proved in \cite{fornaro-lorenzi,metafune-spina-2}.
More specifically, the results in \cite{fornaro-lorenzi} cover the case when $\alpha\in (1,2]$ and show that the realization $A_p$ in $L^p(\R^N)$, with domain
\begin{eqnarray*}
D(A_p)=\left\{u\in W^{2,p}(\R^N): a|D^2u|,\,a^{1/2}|\nabla u|\in L^p(\R^N)\right\},
\end{eqnarray*}
generates a strongly continuous analytic semigroup. For $\alpha>2$, the generation results depend upon $N$ as it is
proved in \cite{metafune-spina-2}. More specifically, if $N=1,2$ no realization of ${\mathcal A}$ in $L^p(\R^N)$ generates a strongly continuous (resp. analytic) semigroup. The same happens if $N\ge 3$ and $p\le N/(N-2)$. On the other hand, if $N\ge 3$, $p>N/(N -2)$ and $2 <\alpha\le(p-1)(N -2)$, then the maximal realization $A_p$ of the operator ${\mathcal A}$ in $L^p(\R^N)$
generates a positive semigroup of contractions, which is also analytic if $\alpha< (p - 1)(N - 2)$.

Here, we confine ourselves to the case when $\alpha\in [0,2]$.
In the first main result of the paper we prove that, for any $1<p<\infty$, the realization $A_p$ of $\mathcal{A}$ in $L^p(\R^N)$, with domain
\begin{eqnarray*}
D(A_p)=\left\{u\in W^{2,p}(\R^N): a|D^2u|,\,a^{1/2}|\nabla u|,\, V u\in L^p(\R^N)\right\},
\end{eqnarray*}
generates a positive strongly continuous and analytic semigroup $(T_p(t))_{t\ge 0}$ for any $\beta\ge 0$. This semigroup is also consistent, irreducible and ultracontractive.
We then show that, if $\beta >0$, $T_p(t)$ is compact for all $t>0$ and the spectrum $\sigma(A_p)$ is independent of $p$.

Due to the local regularity of the coefficients, the semigroup $(T_p(t))_{t\ge 0}$ admits a heat kernel $k(t,x,y)$.
If we denote by $p(t,x,y)$ the heat kernel corresponding to the operator $B=\Delta -|x|^\beta$, then it is known that, for $\beta>2$,
\begin{eqnarray*}
p(t,x,y)\le Ce^{ct^{-b}} \frac{1}{(|x| |y|)^{\frac{\beta}{4}+\frac{N-1}{2}}}e^{-\frac{|x|^{\gamma}}{\gamma}} e^{-\frac{|y|^{\gamma}}{\gamma}},\qquad\;\, 0<t\le 1,
\end{eqnarray*}
for large $|x|$ and $|y|$, $\gamma =1+\frac{\beta}{2}$, and $b>\frac{\beta +2}{\beta -2}$ (see
\cite[Cor. 4.5.5 and Cor. 4.5.8]{davies}).\\
By providing upper and lower estimates for the ground state of $A_p$ corresponding to the largest eigenvalue $\lambda_0$ and adapting the arguments used in \cite{davies}, we show
the heat kernel estimates
\begin{equation}
0<k(t,x,y)\le \frac{Ke^{\lambda_0 t}e^{ct^{-b}}f_0(x)f_0(y)}{1+|y|^\alpha},\qquad\;\, t>0,
\label{estimates-k}
\end{equation}
for $|x|,\,|y|>1$ where
\begin{eqnarray*}
f_0(x):=|x|^{\frac{\alpha-\beta}{4}-\frac{N-1}{2}}\exp\left(-\int_1^{|x|}\frac{s^{\beta/2}}{(1+s^\alpha)^{1/2}}\,ds\right),
\end{eqnarray*}
provided that $\alpha \in [0,2)$ and $\beta >2$.
Such estimates allow us to describe the behaviour of all eigenfunctions of $A_p$ at infinity.

Finally, thanks to a recent technique developed in \cite{ouhabaz-rhandi}, we extend estimates \eqref{estimates-k} to more general elliptic operators in divergence form.

We stress that, in the case where $V\equiv 0$, kernel estimates similar to \eqref{estimates-k} have been
obtained in \cite{metafune-spina-3}, even for $\alpha \ge 2$. We also quote \cite{metafune-spina-1}
where upper and lower estimates for the kernel $k$ have been proved in the case where $\alpha=0$ and $\beta<2$.
\medskip

The paper is structured as follows. In Section \ref{sect-2} we prove the generation results and exploit some of peculiar properties of the semigroup $(T_p(t))_{t\ge 0}$.
Then, in Section \ref{sect-3} we prove upper estimates for the kernel $k$ associated to operator $\mathcal{A}$, when $\alpha\in [0,2)$ and $\beta>2$, and we use them to estimate the asymptotic behaviour of the eigenvalues of the operator $A_p$. Then, we extend
the heat kernel estimates to a more general class of elliptic operators in divergence form.
Finally, in the appendix, we collect two technical results which are used in Section \ref{sect-2}.

\begin{notation}
For any $k\in\N$ (eventually $k=\infty$) we denote by $C^{k}_c(\R^N)$ the set of all
functions $f:\R^N\to\R$ that are continuously differentiable in $\R^N$ up to $k$-th order
and have compact support (say ${\rm supp}(f)$).
Moreover, for any bounded function $f:\R^N\to
\R$ we denote by $\|f\|_{\infty}$ its sup-norm, i.e., $\|f\|_{\infty}=\sup_{x\in\R^N}|f(x)|$.
If $f$ is smooth enough we set
\begin{eqnarray*}
|\nabla f(x)|^2=\sum_{i=1}^N|D_if(x)|^2,\qquad
|D^2f(x)|^2=\sum_{i,j=1}^N|D_{ij}f(x)|^2.
\end{eqnarray*}

For any $x_0\in\R^N$ and any $r>0$ we denote by $B_r(x_0)\subset\R^N$ the open ball, centered at $x_0$ with radius $r$.
We simply write $B_r$ when $x_0=0$.
$\chi_E$ denotes the characteristic function of the (measurable) set $E$,
i.e., $\chi_E(x)=1$ if $x\in E$, $\chi_E(x)=0$ otherwise.

For any $p\in [1,\infty)$ and any positive measure $d\mu$, we simply write $L^p_{\mu}$ instead of $L^p(\R^N,d\mu)$.
The Euclidean inner product in $L^2_{\mu}$ is denoted by $(\cdot,\cdot)_{\mu}$.
In the particular case when $\mu$ is the Lebesgue measure, we keep the classical notation $L^p(\R^N)$ for any $p\in [1,\infty)$.
Finally, by $x\cdot y$ we denote the Euclidean scalar product of the vectors $x,y\in\R^N$.
\end{notation}

\section{Generation results}
\label{sect-2}
For any $\alpha\in [0,2]$ and any $\beta\ge 0$ we denote by $A_p$ the realization
in $L^p(\R^N)$ ($p\in (1,\infty)$) of the operator ${\mathcal A}$, defined in \eqref{operatore-A},  with domain
\begin{align*}
D(A_p)=\left\{u\in W^{2,p}(\R^N): a|D^2u|,\,a^{1/2}|\nabla u|,\, Vu\in L^p(\R^N)\right\}.
\end{align*}
We endow $D(A_p)$ with the norm
\begin{align}
\|u\|_{D(A_p)}=\|u\|_{L^p(\R^N)}+\|Vu\|_{L^p(\R^N)}+\|a^{1/2}|\nabla u|\,\|_{L^p(\R^N)}+\|a|D^2u|\,\|_{L^p(\R^N)},
\label{norma-DAp}
\end{align}
for any $u\in D(A_p)$.

\subsection{Preliminary results and apriori estimates}
This subsection contains all the technical results that we need to prove the generation results in Theorem \ref{prop-2}.

\begin{proposition}{\cite[Prop. 6.1]{cupini-fornaro}}
\label{prop-sim}
Let $\mathcal F=\{B_{\rho(x)}(x): x \in \R^N\}$ be a covering of $\R^N$, where $\rho:\R^N\to\R_+$
is a Lipschitz continuous function, with Lipschitz constant $\kappa$ strictly less
than $1/2$. Then, there exist a countable subcovering
$\{B_{\rho(x_n)}(x_n): n\in\N\}$ and a natural number $\zeta=\zeta(N,\kappa)$
such that at most $\zeta$ among the doubled balls
$\{B_{2\rho(x_n)}(x_n): n\in\N\}$ overlap.
\end{proposition}

\begin{remark}
\label{rem-1}
The previous proposition can be rephrased in terms of characteristic functions as follows: there exist a sequence $(x_n)$ and a natural number $\zeta$
such that
\begin{eqnarray*}
1\le \sum_{n=1}^{\infty}\chi_{B_{\rho(x_n)}(x_n)}(x)\le \sum_{n=1}^{\infty}\chi_{B_{2\rho(x_n)}(x_n)}(x)\le\zeta,\qquad\;\,x\in\R^N.
\end{eqnarray*}
\end{remark}

\begin{proposition}
\label{prop-1}
Fix $p\in (1,\infty)$, and let $q,W:\R^N\to\R$ be two functions with the following properties:
\begin{enumerate}[\rm (i)]
\item
$q\in C(\R^N)\cap C^1(\R^N\setminus\{0\})$ and there exist two positive constants $r$ and $\kappa$ such that
$|\nabla q|\le\kappa q^{1/2}$ in $\R^N\setminus B_r$. Further, $q(x)\ge q_0>0$ for any $x\in\R^N$;
\item
$W\in C(\R^N)\cap C^1(\R^N\setminus\{0\})$, $W(x)\ge w_0>0$ for any $x\in\R^N$ and there exist two constants $c_{1,p}>0$ and $c_{2,p}\in (0,4/(p-1))$ such that
\begin{equation}
|\nabla \Xi(x)|^2\le c_{1,p}\Xi(x)^2+c_{2,p}\Xi(x)^3,\qquad\;\, |x|\ge r,
\label{stima-okazawa-0}
\end{equation}
where $\Xi=q^{-1}W$. Further $0<\xi_0:=\inf_{x\in\R^N}\Xi(x)$.
\end{enumerate}
Then, there exist three positive constants $\varepsilon_0$, $C$ $($depending on $\kappa$, $c_{1,p}$, $c_{2,p}$, $\xi_0$, as well as on $\|q\|_{C^1(B_{2r}\setminus B_r)}$, $\|W\|_{C^1(B_{2r}\setminus B_r)})$ and $C_{\varepsilon}$ $($depending also on $\varepsilon$ and blowing up as $\varepsilon\to 0^+)$ such that
\begin{equation}
\|q^{\frac{1}{2}}|\nabla u|\,\|_{L^p(\R^N)}\le \varepsilon
\|q\Delta u-Wu\|_{L^p(\R^N)}+C_{\varepsilon}\|u\|_{L^p(\R^N)}
\label{stima-interp-1}
\end{equation}
and
\begin{equation}
\|q|D^2u|\,\|_{L^p(\R^N)}\le C\left (\|u\|_{L^p(\R^N)}+\|q\Delta u-Wu\|_{L^p(\R^N)}\right ),
\label{stima-2}
\end{equation}
for any $\varepsilon\in (0,\varepsilon_0]$ and any $u\in W^{2,p}(\R^N)$ with
$q^{1/2}|\nabla u|$, $q|D^2u|$, $Wu\in L^p(\R^N)$.
\end{proposition}

\begin{proof}
In view of Proposition \ref{lemma-1} we can limit ourselves to proving \eqref{stima-interp-1} and \eqref{stima-2} when
$u\in C^{\infty}_c(\R^N)$.
Being rather long, we split the proof into some steps. Throughout the proof $p$ is arbitrarily fixed in $(1,\infty)$.

{\em Step 1.} Let us prove that
\begin{equation}
\|\,|D^2u|\,\|_{L^p(\R^N)}\le C_p\|\Delta u-\tilde\Xi u\|_{L^p(\R^N)},
\label{hessiano-1}
\end{equation}
for any $u\in C^{\infty}_c(\R^N)$ and some positive constant $C_p$. Here,
$\tilde\Xi=\varphi+(1-\varphi)\Xi$, where $\varphi$ is any smooth function such that $\chi_{B_{\tilde r}}\le\varphi\le\chi_{B_{2\tilde r}}$,
and $\tilde r=\max\{r,1\}$.
Note that $\tilde\Xi(x)\ge\tilde\xi_0:=\min\{1,\xi_0\}$ for any $x\in\R^N$. We claim that
\begin{equation}
|\nabla\tilde\Xi(x)|^2\le c_{1,p}'\tilde \Xi(x)^2+c_{2,p}'\tilde \Xi(x)^3,\qquad\;\,x\in \R^N,
\label{stima-okazawa-2}
\end{equation}
for some constants $c_{2,p}'\in (0,4/(p-1))$ and $c_{1,p}'>0$.
Clearly, estimate \eqref{stima-okazawa-2} holds true in $\R^N\setminus B_{2\tilde r}$, by virtue of \eqref{stima-okazawa-0}, and in $B_{\tilde r}$, since $\nabla\Xi$ therein identically vanishes.
For any $x\in B_{2\tilde r}\setminus B_{\tilde r}$ we can estimate
\begin{align*}
|\nabla\tilde\Xi(x)|^2=&|\nabla\varphi(x)(1-\Xi(x))+(1-\varphi(x))\nabla\Xi(x)|^2\\
\le & 2(\|\,|\nabla\varphi|\,\|_{\infty}^2\|\Xi-1\|_{L^{\infty}(B_{2\tilde r}\setminus B_{\tilde r})}^2+\|\,|\nabla\Xi|\,\|_{L^{\infty}(B_{2\tilde r}\setminus B_{\tilde r})}^2)
\tilde\xi_0^{-2}\tilde\Xi(x)^2.
\end{align*}
Hence, inequality \eqref{stima-okazawa-2} follows in the whole of $\R^N$ with $c_{2,p}'=c_{2,p}$ and
\begin{eqnarray*}
c_{1,p}'=\max\{c_{1,p}, 2(\|\,|\nabla\varphi|\,\|_{\infty}^2\|\Xi-1\|_{L^{\infty}(B_{2\tilde r}\setminus B_{\tilde r})}^2+\|\,|\nabla\Xi|\,\|_{L^{\infty}(B_{2\tilde r}\setminus B_{\tilde r})}^2)\tilde\xi_0^{-2}\}.
\end{eqnarray*}
Thanks to \eqref{stima-okazawa-2} we can apply \cite[Lemma 1.4, Thm. 2.1]{okazawa-2}, which yield
\begin{equation}
\|\Delta u\|_{L^p(\R^N)}\le \frac{4}{4-(p-1)c_{2,p}}\|\Delta u-\tilde \Xi u\|_{L^p(\R^N)}+\left (\frac{(p-1)c_{1,p}'}{4-(p-1)c_{2,p}}\right )\|u\|_{L^p(\R^N)}.
\label{delta}
\end{equation}

To complete the proof of estimate \eqref{hessiano-1} we observe that, in view of \eqref{delta} and the well-known Calderon-Zygmund inequality (see e.g., \cite[Thm. 9.19]{gilbarg}), it suffices to show that, for any $p\in (1,\infty)$, there exists a positive
constant $c_p$, independent of $u$, such that
\begin{equation}
\|u\|_{L^p(\R^N)}\le c_p\|\Delta u-\tilde\Xi u\|_{L^p(\R^N)}.
\label{invert}
\end{equation}
Set $f:=\Delta u-\tilde\Xi u$ and multiply both sides of the equation $\Delta u-\tilde\Xi u=f$ by $u|u|^{p-2}$, where we assume without loss of generality that $u$ is real. Indeed, if $u$ is complex-valued, \eqref{invert} will follow arguing on its real and imaginary parts.
Now, for any $f,g\in L^p(\R^N)$ we can estimate
\begin{align*}
\left (\|f\|_{L^p(\R^N)}+\|g\|_{L^p(\R^N)}\right )^p\le & 2^{p-1}\int_{\R^N}(|f|^p+|g|^p)dx\\
\le
& \max\{2^{p/2},2^{p-1}\}\int_{\R^N}(|f|^2+|g|^2)^{\frac{p}{2}}dx.
\end{align*}
We thus get \eqref{invert} with $c_p$ being replaced by
$c_p\max\{\sqrt{2},2^{1-1/p}\}$.

A straightforward computation, based on an integration by parts (in the case when $p\ge 2$) and on \cite[Thm. 3.1]{metafune-spina-0} (in the case when $p\in (1,2)$) shows that
\begin{eqnarray*}
-(p-1)\int_{\R^N}|u|^{p-2}|\nabla u|^2\chi_{\{u\neq 0\}}dx-\tilde\xi_0\int_{\R^N}|u|^pdx\ge\int_{\R^N}fu|u|^{p-2}dx,
\end{eqnarray*}
which yields \eqref{invert} with $c_p=\tilde\xi_0^{-1}$.

{\em Step 2.}
Let us now prove estimate \eqref{stima-interp-1} by a covering argument. The starting point is the well-known interpolation inequality (see e.g., \cite{triebel})
\begin{equation}
\|\,|\nabla v|\,\|_{L^p(\R^N)}\le c_N\|v\|_{L^p(\R^N)}^{\frac{1}{2}}\|\,|D^2v|\,\|_{L^p(\R^N)}^{\frac{1}{2}},\qquad\;\,v\in W^{2,p}(\R^N),
\label{interp}
\end{equation}
which, in view of \eqref{hessiano-1}, allows us to estimate
\begin{equation}
\|\,|\nabla u|\,\|_{L^p(\R^N)}\le c_{N,p}'\|\Delta u-\tilde\Xi u\|_{L^p(\R^N)}^{\frac{1}{2}}\|u\|_{L^p(\R^N)}^{\frac{1}{2}},
\label{1}
\end{equation}
for some positive constant $c_{N,p}'$, independent of $u\in C^{\infty}_c(\R^N)$.

We can now apply the covering argument to estimate \eqref{1}. For this purpose, let
$\tilde q:=2^{-4}\varphi+(1-\varphi)q$, where $\varphi$ is as above.
Arguing as in Step 1 we can easily show that $|\nabla\tilde q(x)|\le\tilde\kappa \tilde q(x)^{1/2}$ for any $x\in\R^N$,
where
\begin{eqnarray*}
\tilde\kappa=\max\{1,\kappa,q_0^{-1/2}(\|\,|\nabla\varphi|\,\|_{\infty}\|q-2^{-4}\|_{L^{\infty}(B_{2\tilde r}\setminus B_{\tilde r})}+\|\,|\nabla q|\,\|_{L^{\infty}(B_{2\tilde r}\setminus B_{\tilde r})})\}.
\end{eqnarray*}
Further, we introduce the function $\rho:\R^N\to\R$ defined by
\begin{eqnarray*}
\rho(x)=\frac{1}{2\tilde\kappa}\tilde q(x)^{\frac{1}{2}},\qquad\;\,x\in\R^N.
\end{eqnarray*}
Clearly, $\rho$ is a Lipschitz continuous function, with Lipschitz constant not greater than $1/4$.
Moreover,
\begin{equation}
\frac{1}{2}\tilde q(x_0)^{\frac{1}{2}}\le \tilde q(x)^{\frac{1}{2}}\le \frac{3}{2}\tilde q(x_0)^{\frac{1}{2}},\qquad\;\,x\in B_{2\rho(x_0)}(x_0)
\label{rho-1}
\end{equation}
and
\begin{eqnarray*}
\rho(x)\le \frac{1}{4}|x|+\frac{1}{8},\qquad\;\, x\in\R^N.
\end{eqnarray*}
This latter inequality implies that
\begin{equation}
B_{\rho(x_0)}(x_0)\subset B_{(1+10|x_0|)/8},\qquad\;\,
B_{2\rho(x_0)}(x_0)\subset\R^N\setminus B_{-1/4+|x_0|/2},
\label{rho-2}
\end{equation}
for any $|x_0|\ge 1/2$.

Now, for any $x_0\in\R^N$ we set $\vartheta_{x_0}(x)=\vartheta(\frac{x-x_0}{\rho(x_0)})$, where
$\vartheta \in C_c^\infty(\R^N)$ satisfies $\chi_{B_1}\le
\vartheta\le \chi_{B_2}$. Moreover, to fix the notation, we set
$L:=\|\,|\nabla\vartheta|\,\|_{\infty}+\|\Delta\vartheta\|_{\infty}$.
Applying estimate \eqref{1}
to the function $\vartheta_{x_0} u$ and using Young inequality, we get
\begin{align*}
&\|\tilde q(x_0)^{\frac{1}{2}}|\nabla u|\,\|_{L^p(B_{\rho(x_0)}(x_0))}\\
\le &\|\tilde q(x_0)^{\frac{1}{2}}|\nabla
(\vartheta_{x_0} u)|\,\|_{L^p(\R^N)}\\
\le & c_{N,p}' \| \vartheta_{x_0}u\|_{L^p(\R^N)}^{\frac{1}{2}}\|\tilde q(x_0)\Delta(\vartheta_{x_0}
u)-\tilde q(x_0)\tilde\Xi\vartheta_{x_0}u\|_{L^p(\R^N)}^{\frac{1}{2}}\\
\le& c_{N,p}' \bigg(\varepsilon\| \tilde q(x_0)\Delta(\vartheta_{x_0}
u)-\tilde q(x_0)\tilde\Xi\vartheta_{x_0}u\|_{L^p(\R^N)}+\frac{1}{4\varepsilon}\| \vartheta_{x_0} u\|_{L^p(\R^N)}\bigg)\\
\le &c_{N,p}'\bigg(\varepsilon\| \tilde q(x_0)\Delta
u-\tilde q(x_0)\tilde\Xi u\|_{L^p(B_{2\rho(x_0)}(x_0))}+\frac{2L}{{\rho(x_0)}}\varepsilon\|\tilde q(x_0)|\nabla
u|\,\|_{L^p(B_{2\rho(x_0)}(x_0))}\\
&\qquad\;\,
+\frac{L}{{\rho(x_0)}^2}\varepsilon\|\tilde q(x_0)\,u\|_{L^p(B_{2\rho(x_0)}(x_0))}+\frac{1}{4\varepsilon}\|u\|_{L^p(B_{2\rho(x_0)}(x_0))}\bigg),
\end{align*}
for any $\varepsilon>0$. Now, from \eqref{rho-1} we deduce that
\begin{align*}
&\bullet~\|\tilde q(x_0)\Delta u-\tilde q(x_0)\tilde\Xi u\|_{L^p(B_{2\rho(x_0)}(x_0))}\le 4\|\tilde q\Delta u-\tilde q \tilde \Xi u\|_{L^p(B_{2\rho(x_0)}(x_0))};\\[2mm]
&\bullet~\frac{2L}{\rho(x_0)}\|\tilde q(x_0)|\nabla u|\,\|_{L^p(B_{2\rho(x_0)}(x_0))}
=  4L\tilde\kappa\|\tilde q(x_0)^{\frac{1}{2}}|\nabla u|\,\|_{L^p(B_{2\rho(x_0)}(x_0))}\\
&\qquad\qquad\qquad\qquad\qquad\qquad\qquad\;\;\;\,\le  8L\tilde\kappa\|\tilde q^{\frac{1}{2}}|\nabla u|\,\|_{L^p(B_{2\rho(x_0)}(x_0))};\\[2mm]
&\bullet~\frac{L}{{\rho(x_0)}^2}\|\tilde q(x_0)u\|_{L^p(B_{2\rho(x_0)}(x_0))}=  4L\tilde\kappa^2\|u\|_{L^p(B_{2\rho(x_0)}(x_0))}.
\end{align*}
Hence,
\begin{align*}
&\|\tilde q^{\frac{1}{2}}|\nabla u|\,\|_{L^p(B_{\rho(x_0)}(x_0))}\notag\\
\le & \frac{3}{2}\|\tilde q(x_0)^{\frac{1}{2}}|\nabla u|\,\|_{L^p(B_{\rho(x_0)}(x_0))}\nonumber\\
\le &6c_{N,p}'\bigg\{\varepsilon\|\tilde q\,\Delta u-\tilde q\tilde\Xi u\|_{L^p(B_{2\rho(x_0)}(x_0) )}
+2\varepsilon L\tilde\kappa\|\tilde q^{\frac{1}{2}}|\nabla u|\,\|_{L^p(B_{2\rho(x_0)}(x_0))}\nonumber\\
&\qquad\quad+\bigg (\varepsilon L\tilde\kappa^2+\frac{1}{16\varepsilon}\bigg )\|u\|_{L^p(B_{2\rho(x_0)}(x_0))}\bigg\}.
\end{align*}
By Proposition \ref{prop-sim} there exist a sequence $(x_n)$ and a positive number $\zeta$ such that ${\mathcal F}'=\{B_{\rho(x_n)}(x_n): n\in\N\}$ is a covering of $\R^N$ and the intersection of more than $\zeta$ balls from ${\mathcal F}'$ is empty. From \eqref{rho-2} it is immediate to conclude that $\{B_{\rho(x_n)}(x_n): |x_n|\ge 6\tilde r\}$ is a covering of $\R^N\setminus B_{8\tilde r}$ and
$B_{2\rho(x_n)}(x_n)\subset \R^N\setminus B_{2\tilde r}$ for $|x_n|\ge 6\tilde r$ and any $n\in\N$. Taking Remark \ref{rem-1} into account and recalling that $\tilde q=q$
and $\tilde q\tilde\Xi=W$ in $\R^N\setminus B_{2\tilde r}$, we can write
\begin{align*}
\|q^{\frac{1}{2}}|\nabla u|\,\|_{L^p(\R^N\setminus B_{8\tilde r})}^p
\le & c_{N,p}''\varepsilon^p\int_{\R^N}|q\Delta u-Wu|^p\sum_{|x_n|\ge 6\tilde r}\chi_{B_{2\rho(x_n)}(x_n)}dx\notag\\
&+c_{N,p}''(\varepsilon L\tilde\kappa)^p\int_{\R^N}q^{\frac{p}{2}}|\nabla u|^p\sum_{|x_n|\ge 6\tilde r}\chi_{B_{2\rho(x_n)}(x_n)}dx\notag\\
&+c_{N,p}''\bigg (\varepsilon L\tilde\kappa^2+\frac{1}{16\varepsilon}\bigg )^p\int_{\R^N}|u|^p\sum_{|x_n|\ge 6\tilde r}\chi_{B_{2\rho(x_n)}(x_n)}dx\notag\\
\le & c_{N,p}''\zeta\varepsilon^p\|q\Delta u-Wu\|_{L^p(\R^N\setminus B_{2\tilde r})}^p\notag\\
&+c_{N,p}''\zeta(\varepsilon L\tilde\kappa)^p\|q^{\frac{1}{2}}|\nabla u|\,\|_{L^p(\R^N\setminus B_{2\tilde r})}^p\notag\\
&+c_{N,p}''\zeta\bigg (\varepsilon L\tilde\kappa^2+\frac{1}{16\varepsilon}\bigg )^p\|u\|_{L^p(\R^N\setminus B_{2\tilde r})}^p.
\end{align*}
Due to the arbitrariness of $\varepsilon>0$, from the above estimate we get
\begin{align}
\|q^{\frac{1}{2}}|\nabla u|\,\|_{L^p(\R^N\setminus B_{8\tilde r})}
\le & \varepsilon \|q\Delta u-Wu\|_{L^p(\R^N)}
+\varepsilon\|q^{\frac{1}{2}}|\nabla u|\,\|_{L^p(\R^N)}+C_{\varepsilon}\|u\|_{L^p(\R^N)},
\label{stima-outer}
\end{align}
for any $\varepsilon>0$ and some positive constant $C_{\varepsilon}$, possibly blowing up as $\varepsilon\to 0^+$.

To extend the previous inequality to the whole of $\R^N$ we
use the classical interior $L^p$-estimates (see e.g.,
\cite[Thm. 9.11]{gilbarg})
\begin{align}
\|u\|_{W^{2,p}(B_{8\tilde r})}
\le &K_1(\|q\Delta u-Wu\|_{L^p(\R^N)}+\|v\|_{L^p(\R^N)}),
\label{apriori}
\end{align}
and the interpolative estimate
\begin{eqnarray*}
\|\,|\nabla v|\,\|_{L^p(B_{8\tilde r})}\le K_2\|v\|_{L^p(B_{8\tilde r})}^{\frac{1}{2}}\|v\|_{W^{2,p}(B_{8\tilde r})}^{\frac{1}{2}},\qquad\;\,
v\in W^{2,p}(B_{8\tilde r}),
\end{eqnarray*}
which hold for some positive constants $K_1$ and $K_2$ independent of $u$ and $v$, respectively,
to infer that
\begin{align}
\|q^{\frac{1}{2}}|\nabla u|\,\|_{L^p(B_{8\tilde r})}\le &
\|q^{\frac{1}{2}}\|_{L^{\infty}(B_{8\tilde r})}\|\,|\nabla u|\,\|_{L^p(B_{8\tilde r})}\notag\\
\le &K_3\|u\|_{L^p(B_{8\tilde r})}^{\frac{1}{2}}\left (\|q\Delta u-Wu\|_{L^p(\R^N)}+\|u\|_{L^p(\R^N)}\right )^{\frac{1}{2}}\notag\\
\le &K_3\|u\|_{L^p(B_{8\tilde r})}^{\frac{1}{2}}\|q\Delta u-Wu\|_{L^p(\R^N)}^{\frac{1}{2}}+K_3\|u\|_{L^p(\R^N)}\notag\\
\le & \varepsilon \|q\Delta u-Wu\|_{L^p(\R^N)}+C_{\varepsilon}'\|u\|_{L^p(\R^N)},
\label{stima-inner}
\end{align}
for any $\varepsilon>0$ and some positive constants $K_3$ and $C_{\varepsilon}'$, this latter one possibly blowing up as $\varepsilon\to 0^+$.
From estimates \eqref{stima-outer} and \eqref{stima-inner} we deduce that
\begin{align*}
\|q^{\frac{1}{2}}|\nabla u|\,\|_{L^p(\R^N)}\le &
\|q^{\frac{1}{2}}|\nabla u|\,\|_{L^p(B_{8\tilde r})}+\|q^{\frac{1}{2}}|\nabla u|\,\|_{L^p(\R^N\setminus B_{8\tilde r})}\\
\le & 2\varepsilon \|q\Delta u-Wu\|_{L^p(\R^N)}+C_{\varepsilon}''\|u\|_{L^p(\R^N)}
+\varepsilon\|q^{\frac{1}{2}}|\nabla u|\,\|_{L^p(\R^N)},
\end{align*}
for any $\varepsilon>0$ and some positive constant $C_{\varepsilon}''$ possibly blowing up as $\varepsilon\to 0^+$. Hence,
taking $\varepsilon<1$, we immediately get \eqref{stima-interp-1}.

\emph{Step 3.} To conclude the proof, let us prove estimate \eqref{stima-2}.
From \eqref{hessiano-1} applied to the function $u\vartheta_{x_0}$ we deduce that
\begin{eqnarray*}
\|q(x_0)|D^2(\vartheta_{x_0} u)|\,\|_{L^p(\R^N)}\le C_p\|q(x_0)\Delta (\vartheta_{x_0} u)-q(x_0)\tilde\Xi \vartheta_{x_0} u\|_{L^p(\R^N)}.
\end{eqnarray*}
Therefore, taking \eqref{rho-1} into account, and arguing as in the proof of Step 2, we first get
\begin{align*}
&\|\tilde q|D^2u|\,\|_{L^p(B_{\rho(x_0)}(x_0))}\\
\le & \frac{9}{4}\|\tilde q(x_0)|D^2u|\,\|_{L^p(B_{\rho(x_0)}(x_0))}\\
\le & K_4\Big (\|\tilde q(x_0)\Delta u-\tilde q(x_0)\tilde\Xi u\|_{L^p(B_{2\rho(x_0)}(x_0))}\\
&\qquad+\|\tilde q(x_0)\nabla u\cdot\nabla \vartheta_{x_0}\|_{L^p(\R^N)}+\|\tilde q(x_0)u\Delta\vartheta_{x_0}\|_{L^p(\R^N)}\Big )\\
\le & K_5\bigg (\|\tilde q\Delta u-\tilde q\tilde\Xi u\|_{L^p(B_{2\rho(x_0)}(x_0))}
+\|\tilde q^{\frac{1}{2}}|\nabla u|\,\|_{L^p(B_{2\rho(x_0)}(x_0))}+\|u\|_{L^p(B_{2\rho(x_0)}(x_0))}\bigg )
\end{align*}
and, then, applying the same covering argument as above, we conclude that
\begin{equation}
\|q|D^2u|\,\|_{L^p(\R^N\setminus B_{8\tilde r})}\le K_6\left (\|q\Delta u-Wu\|_{L^p(\R^N)}+\|q^{\frac{1}{2}}|\nabla u|\,\|_{L^p(\R^N)}
+\|u\|_{L^p(\R^N)}\right ).
\label{stima-4}
\end{equation}
Here, $K_4$, $K_5$ and $K_6$ are positive constant, independent of $u$.
Combining \eqref{stima-interp-1}  and \eqref{stima-4}, we get
\begin{equation}
\|q|D^2u|\,\|_{L^p(\R^N\setminus B_{8\tilde r})}\le K_7\left (\|q\Delta u-Wu\|_{L^p(\R^N)}+\|u\|_{L^p(\R^N)}\right ),
\label{stima-5}
\end{equation}
for some positive constant $K_7$, independent of $u$.
Estimate \eqref{stima-2} now follows from \eqref{stima-5} and \eqref{apriori} (with $u$ replacing $v$).
\end{proof}

The following result is now a straightforward consequence of Proposition \ref{prop-1}.

\begin{corollary}
\label{cor-2.6}
The norm of $D(A_p)$ defined in \eqref{norma-DAp} is equivalent to the graph norm of $D(A_p)$, i.e., to the norm defined by $\|u\|_{D(A_p)}=\|u\|_{L^p(\R^N)}+\|A_pu\|_{L^p(\R^N)}$ for any $u\in D(A_p)$.
\end{corollary}

\subsection{Proof of the sectoriality of operator $A_p$}

We can now prove that operator $A_p$ is sectorial in $L^p(\R^N)$ for any $p\in (1,\infty)$.

\begin{theorem}
\label{prop-2}
For any $p\in (1,\infty)$ the operator $A_p$ generates a strongly continuous
analytic semigroup $(T_p(t))_{t\ge 0}$ in $L^p(\R^N)$ which is also positive and consistent.
\end{theorem}

\begin{proof}
Being rather long, we split the proof into several steps.

\emph{Step 1.} For any $\sigma\in (0,1)$ let us introduce the functions $q_{\sigma}$ and $V_{1,\sigma}$ defined by
\begin{eqnarray*}
q_{\sigma}(x)=\frac{1+|x|^{\alpha}}{1+\sigma(1+|x|^{\alpha})},\qquad\;\,
V_{1,\sigma}(x)=\frac{|x|^{\beta}+1}{1+\sigma (|x|^{\beta}+1)},\qquad\;\,
x\in \R^N.
\end{eqnarray*}
As it is immediately seen, $q_{\sigma}$ and $V_{1,\sigma}$ are bounded in $\R^N$ and satisfy
\begin{eqnarray*}
q_{\sigma}(x)\ge \frac{1}{2},\;\,V_{1,\sigma}(x)\ge \frac{1}{2},
\end{eqnarray*}
for any $x\in \R^N$ and any $\sigma \in (0,1)$.
By well-known results (see e.g., \cite[Chpt. 3]{lunardi}), for any $\sigma\in (0,1)$ and any $p\in (1,\infty)$, the realization $Q_{\sigma,p}$ of operator
$\mathcal{Q}_{\sigma}=q_{\sigma}\Delta-V_{1,\sigma}$
in $L^p(\R^N)$, with $W^{2,p}(\R^N)$ as a domain, is the generator of a strongly continuous,
analytic semigroup.
Here we are aimed at proving the estimates
\begin{align}
&\|q_{\sigma}^{\frac{1}{2}}|\nabla u|\,\|_{L^p(\R^N)}\le \varepsilon
\|Q_{\sigma,p}u\|_{L^p(\R^N)}+C_{\varepsilon}\|u\|_{L^p(\R^N)},
\label{stima-a}
\\
&\|q_{\sigma}|D^2u|\,\|_{L^p(\R^N)}\le C\left (\|u\|_{L^p(\R^N)}+\|Q_{\sigma,p}u\|_{L^p(\R^N)}\right ),
\label{stima-b}
\end{align}
for any $\varepsilon>0$, any $u\in W^{2,p}(\R^N)$ and some positive constants $C$ and $C_{\varepsilon}$, independent of $\sigma\in (0,1)$,
the latter constant possibly blowing up as $\varepsilon\to 0^+$.
For this purpose we prove that, for any $\sigma\in (0,1)$, functions $q_{\sigma}$ and $V_{1,\sigma}$ satisfy the assumptions of Theorem \ref{prop-1} with $r=1/2$ and the constants therein appearing being independent of $\sigma$. More specifically, we should establish the following facts:
\begin{enumerate}[\rm (i)]
 \item
there exists a constant $c_1>0$, independent of $\sigma\in (0,1)$, such that
\begin{equation}
|\nabla\Xi_{1,\sigma}(x)|^2\le c_1\Xi_{1,\sigma}(x)^2,\qquad\;\,x\in\R^N\setminus B_{1/2},
\label{okazawa-3}
\end{equation}
where
\begin{eqnarray*}
\Xi_{1,\sigma}(x):=\frac{V_{1,\sigma}(x)}{q_{\sigma}(x)}=
\frac{1+|x|^{\beta}}{1+|x|^{\alpha}}\cdot\frac{1+\sigma (1+|x|^{\alpha})}{1+\sigma(1+|x|^{\beta})},\qquad\;\,x\in\R^N;
\end{eqnarray*}
\item
there exists $\kappa>0$, independent of $\sigma\in (0,1)$, such that
\begin{equation}
|\nabla q_{\sigma}(x)|\le \kappa q_{\sigma}(x)^{\frac{1}{2}},\qquad\;\,x\in\R^N\setminus B_{1/2}.
\label{estim-kappa}
\end{equation}
\end{enumerate}

Let us begin by checking property (i). Note that
\begin{align*}
\nabla \Xi_{1,\sigma}(x)=&
\frac{(\beta-\alpha)|x|^{\alpha+\beta}+\beta |x|^{\beta}-\alpha |x|^{\alpha}}{|x|^2(1+|x|^{\alpha})^2}
\cdot\frac{1+\sigma(1+|x|^{\alpha})}{1+\sigma(1+|x|^{\beta})}x\\
&+\sigma\frac{1+|x|^{\beta}}{1+|x|^{\alpha}}
\cdot\frac{\alpha|x|^{\alpha}(1+\sigma(1+|x|^{\beta}))-\beta|x|^{\beta}(1+\sigma(1+|x|^{\alpha}))}{|x|^2(1+\sigma(1+|x|^{\beta}))^2}x,
\end{align*}
for any $x\in\R^N\setminus\{0\}$. Since
\begin{eqnarray*}
\min\left\{1,\frac{2+|x|^{\alpha}}{2+|x|^{\beta}}\right\}\le \frac{1+\sigma(1+|x|^{\alpha})}{1+\sigma(1+|x|^{\beta})},\qquad\;\,x\in\R^N,\;\,\sigma\in (0,1),
\end{eqnarray*}
it follows easily that $\Xi_{1,\sigma}(x)\ge 1/2$ for any $x\in B_1$ and any $\sigma\in (0,1)$.
Hence,
\begin{align*}
|\nabla\Xi_{1,\sigma}(x)|\le 4|\beta-\alpha|+6\alpha+8\beta
\le 4(2|\beta-\alpha|+3\beta+4\alpha)\Xi_{1,\sigma}(x),
\end{align*}
for any $x\in B_1\setminus B_{1/2}$ and
\begin{align*}
|\nabla\Xi_{1,\sigma}(x)|\le &
\left (|\beta-\alpha|\frac{|x|^{\beta-1}}{1+|x|^{\alpha}}+\beta\frac{|x|^{\beta-1}}{1+|x|^{\alpha}}+
\frac{\alpha}{|x|(1+|x|^{\alpha})}\right )\frac{1+\sigma(1+|x|^{\alpha})}{1+\sigma(1+|x|^{\beta})}\\
&+\alpha\frac{1+|x|^{\beta}}{1+|x|^{\alpha}}\cdot\frac{1}{|x|}\cdot\frac{\sigma|x|^{\alpha}}{1+\sigma(1+|x|^{\beta})}
+\beta\frac{1+|x|^{\beta}}{1+|x|^{\alpha}}\cdot\frac{1}{|x|}\cdot\frac{1+\sigma(1+|x|^{\alpha})}{1+\sigma(1+|x|^{\beta})}\\
\le & (|\beta-\alpha|+2\alpha+2\beta)\frac{\Xi_{1,\sigma}(x)}{|x|}\\
\le & (|\beta-\alpha|+2\alpha+2\beta)\Xi_{1,\sigma}(x),
\end{align*}
if $x\in\R^N\setminus B_1$.
Hence,
\begin{eqnarray*}
|\nabla\Xi_{1,\sigma}(x)|\le 4(2|\beta-\alpha|+3\beta+4\alpha)\Xi_{1,\sigma}(x),\qquad\;\,x\in\R^N\setminus B_{1/2}.
\end{eqnarray*}
Thus, estimate \eqref{okazawa-3} follows with
$c_1=16(2|\beta-\alpha|+3\beta+4\alpha)^2$.

Finally, a straightforward computation shows that \eqref{estim-kappa} holds true with $\kappa=\alpha 2^{1-\frac{\alpha}{2}}$.
Estimates \eqref{stima-a} and \eqref{stima-b} are thus proved.

\emph{Step 2.} Here we prove that, for any $p\in (1,\infty)$, there exist $\omega_0\in\R$ and $M_p>0$ such that
\begin{equation}
|\lambda|\|u\|_{L^p(\R^N)}\le M_p\|\lambda u-Q_{1/n,p}u\|_{L^p(\R^N)},
\label{hille-yosida}
\end{equation}
for any $u\in D(A_p)$, any $\lambda\in\C$ with ${\rm Re}\lambda\ge\omega_0$ and any $n\in\N$.

We begin by considering the case when $\alpha\in [1,2]$. We fix $p\in (1,\infty)$, $\lambda\in\C$, $u\in W^{2,p}(\R^N)$ and set $f:=\lambda u-Q_{p,1/n}u$. We multiply both sides of this
equation by $\overline{u}|u|^{p-2}$ and integrate by parts, taking \cite[Thm. 3.1]{metafune-spina-0} into account. We get
\begin{align}
\int_{\R^N}f\overline{u}|u|^{p-2}dx=&\lambda\int_{\R^N}|u|^pdx+\int_{\R^N}V_{1,1/n}|u|^pdx-\int_{\R^N}q_{1/n}|u|^{p-2}\overline{u}\Delta udx\notag\\
=&\lambda\int_{\R^N}|u|^pdx+\int_{\R^N}V_{1,1/n}|u|^pdx\notag\\
&+(p-1)\int_{\R^N}q_{1/n}|u|^{p-4}|{\rm Re}(\overline{u}\nabla u)|^2\chi_{\{u\neq 0\}}dx\notag\\
&+\int_{\R^N}q_{1/n}|u|^{p-4}|{\rm Im}(\overline{u}\nabla u)|^2\chi_{\{u\neq 0\}}dx\notag\\
&+i(p-2)\int_{\R^N}q_{1/n}|u|^{p-4}
{\rm Re}(\overline{u}\nabla u)\cdot {\rm Im}(\overline{u}\nabla u)\chi_{\{u\neq 0\}}dx\notag\\
&-\int_{\R^N}\nabla q_{1/n}\cdot \nabla u\, |u|^{p-2}\overline{u}dx.
\label{fiore}
\end{align}
Taking the real part of the first and last side of \eqref{fiore} we
get
\begin{align*}
\int_{\R^N}{\rm Re}(f\overline{u})|u|^{p-2}dx
=&{\rm Re}\lambda\int_{\R^N}|u|^pdx+\int_{\R^N}V_{1,1/n}|u|^pdx\notag\\
&+(p-1)\int_{\R^N}q_{1/n}|u|^{p-4}|{\rm Re}(\overline{u}\nabla u)|^2\chi_{\{u\neq 0\}}dx\notag\\
&+\int_{\R^N}q_{1/n}|u|^{p-4}|{\rm Im}(\overline{u}\nabla u)|^2\chi_{\{u\neq 0\}}dx\notag\\
&-\int_{\R^N}\nabla q_{1/n}\cdot {\rm Re}(\overline{u}\nabla u)u|u|^{p-4}\chi_{\{u\neq 0\}}dx.
\end{align*}
For notational convenience we set
\begin{align*}
&A:=\int_{\R^N}q_{1/n}|u|^{p-4}|{\rm Re}(\overline{u}\nabla u)|^2\chi_{\{u\neq 0\}}dx,\\[1mm]
&B:=\int_{\R^N}q_{1/n}|u|^{p-4}|{\rm Im}(\overline{u}\nabla u)|^2\chi_{\{u\neq 0\}}dx.
\end{align*}
Recalling that $|\nabla q_{1/n}|\le 2|q_{1/n}|^{1/2}$ and using H\"older and Young inequalities we can estimate
\begin{align*}
\|f\|_{L^p(\R^N)}\|u\|_{L^p(\R^N)}^{p-1}
\ge& ({\rm Re}\lambda) \|u\|_{L^p(\R^N)}^p+(p-1)A+B\\
&-2\int_{\R^N}q_{1/n}^{\frac{1}{2}}|{\rm Re}(\overline{u}\nabla u)| |u|^{p-4}\overline{u}dx\\
\ge &({\rm Re}\lambda) \|u\|_{L^p(\R^N)}^p+(p-1)A+B
-2A^{\frac{1}{2}}\|u\|_{L^p(\R^N)}^{\frac{p}{2}}\\
\ge & \left ({\rm Re}\lambda-\frac{2}{p-1}\right )\|u\|_{L^p(\R^N)}^p
+\frac{p-1}{2}A+B.
\end{align*}
We thus deduce that
\begin{align}
&\frac{1}{2}({\rm Re}\lambda)\|u\|_{L^p(\R^N)}\le \|f\|_{L^p(\R^N)},
\label{A}
\\[1mm]
&A\le \left (\frac{p-1}{2}\right )^{p-2}\|f\|_{L^p(\R^N)},
\label{B}
\\[1mm]
&B\le \left (\frac{p-1}{2}\right )^{p-1}\|f\|_{L^p(\R^N)},
\label{C}
\end{align}
for any $\lambda\in\C$ with ${\rm Re}\lambda\ge 4(p-1)^{-1}$. Now, taking the imaginary part of \eqref{fiore}, we get
\begin{align*}
\int_{\R^N}{\rm Im}(f\overline{u})|u|^{p-2}dx
=&({\rm Im}\lambda)\int_{\R^N}|u|^pdx\\
&+(p-2)\int_{\R^N}q_{1/n}|u|^{p-4}{\rm Re}(\overline{u}\nabla u)\cdot {\rm Im}(\overline{u}\nabla u)\chi_{\{u\neq 0\}}dx\notag\\
&-\int_{\R^N}\nabla q_{1/n}\cdot {\rm Im}(\overline{u}\nabla u)u|u|^{p-4}\chi_{\{u\neq 0\}}dx.
\end{align*}
Hence,
\begin{align*}
|{\rm Im}\lambda|\|u\|_{L^p(\R^N)}
\le & \|f\|_{L^p(\R^N)}\|u\|_{L^p(\R^N)}+|p-2|A^{\frac{1}{2}}B^{\frac{1}{2}}+2A^{\frac{1}{2}}
\|u\|_{L^p(\R^N)}^{\frac{p}{2}}.
\end{align*}
Using \eqref{A}-\eqref{C} we obtain
\begin{align}
|{\rm Im}\lambda|\|u\|_{L^p(\R^N)}\le C_p\|f\|_{L^p(\R^N)},\qquad\;\,\lambda\in\C,\;\,{\rm Re}\lambda\ge\frac{4}{p-1}.
\label{imaginary}
\end{align}
From \eqref{A} and \eqref{imaginary}, estimate \eqref{hille-yosida} follows at once with $\omega_0=4(p-1)^{-1}$.

In the case when $\alpha\in [0,1)$, the function $q_{\sigma}$ does not belong to $W^{1,\infty}(\R^N)$
and, consequently, we cannot control $\nabla q_{1/n}$ by $q_{1/n}^{1/2}$ since the gradient of $q_{1/n}$
blows up as $x$ tends to $0$.

To prove estimate \eqref{hille-yosida} we regularize the function $q_{1/n}$ in a neighborhood of the origin by introducing the function
$\hat q_{1/n}:=\varphi+(1-\varphi)q_{1/n}$, where $\varphi$ is a smooth function such that $\chi_{B_1}\le\varphi\le\chi_{B_2}$.
The arguments used above apply to the realization $\hat Q_{1/n,p}$ of the operator $\hat{\mathcal A}=\hat q_{1/n}\Delta-V$ in $L^p(\R^N)$, with $W^{2,p}(\R^N)$ as a domain, since $|\nabla\hat q_{1/n}|\le\hat\kappa\hat q_{1/n}^{1/2}$ in $\R^N$, for some positive constant $\hat\kappa$ independent of $n$. We thus deduce that there exist $\hat\omega_0>0$ and $\hat M_p>0$ such that
\begin{eqnarray*}
|\lambda|\|u\|_{L^p(\R^N)}\le\hat M_p\|\lambda u-\hat Q_{1/n,p}u\|_{L^p(\R^N)},
\end{eqnarray*}
for any $u\in W^{2,p}(\R^N)$, any $\lambda\in\C$ with ${\rm Re}\lambda\ge\hat\omega_0$ and any $n\in\N$. Since $\hat q_{1/n}=q_{1/n}$ in $\R^N\setminus B_2$, we can estimate
\begin{align}
|\lambda|\|u\|_{L^p(\R^N)}\le &\|\lambda u-\hat Q_{1/n,p}u\|_{L^p(B_2)}+\|\lambda u-Q_{1/n,p}u\|_{L^p(\R^N\setminus B_2)}\notag\\
\le &\|(\hat q_{1/n}-q_{1/n})\Delta u\|_{L^p(B_2)}+\|\lambda u-Q_{1/n,p}u\|_{L^p(B_2)}\notag\\
& +\|\lambda u-Q_{1/n,p}u\|_{L^p(\R^N\setminus B_2)}\notag\\
\le &\|\hat q_{1/n}-q_{1/n}\|_{\infty}\|\Delta u\|_{L^p(B_2)}+2\|\lambda u-Q_{1/n,p}u\|_{L^p(\R^N)}\notag\\
\le &4\|\Delta u\|_{L^p(B_2)}+2\|\lambda u-Q_{1/n,p}u\|_{L^p(\R^N)}.
\label{mart-0}
\end{align}
We now apply estimate \eqref{stima-interna-hessiano} with $r=2$ to the operator $L_n=\tilde q_{1/n}\Delta-\tilde V$, where
$\tilde q_{1/n}=q_{1/n}\psi+1-\psi$, $\tilde V=\psi V$, and $\psi$ is any smooth function such that $\chi_{B_4}\le\psi\le\chi_{B_8}$.
Note that the sup-norm and the modulus of continuity of the function $\tilde q_{1/n}$ can be estimated independently of $n$.
So, we can determine two positive constants $\tilde\omega_{0,p}$ and $K_p$, independent of $n$ and $u$, such that
\begin{align*}
\|\Delta u\|_{L^p(B_2)}\le &K_p\left (\|\lambda u-Q_{1/n,p}u\|_{L^p(B_4)}
+\|u\|_{L^p(B_4)}\right )\\
\le &K_p\left (\|\lambda u-Q_{1/n,p}u\|_{L^p(\R^N)}+\|u\|_{L^p(\R^N)}\right ),
\end{align*}
for any $\lambda\in\C$ with ${\rm Re}\lambda\ge\tilde\omega_0$,
which replaced in \eqref{mart-0} yields \eqref{hille-yosida}.

\emph{Step 3.} Here, we fix $p\in (1,\infty)$ and prove that the equation
$\lambda u-A_pu=f$ admits a unique solution $u\in D(A_p)$ for any $f\in L^p(\R^N)$ and any $\lambda\in\C$ with real part not less than $\omega_0+1$.
As a first step, we observe that $\rho(Q_{1/n,p})\supset \Sigma:=\{\lambda\in\C: {\rm Re}\lambda>\omega_0\}$ for any $n\in\N$.
Indeed, as we have already remarked, $Q_{1/n,p}$ is a sectorial operator; hence,  its resolvent set contains a right-halfline.
Such a right-halfline contains $\Sigma$.
Indeed, it is well-known that the function $\lambda\mapsto\|R(\lambda,Q_{1/n,p})\|_{L(L^p(\R^N))}$ blows up as $\lambda$ tends to the boundary
of $\rho(Q_{1/n,p})$, and \eqref{hille-yosida} shows that this cannot be case at any point of $\Sigma$.

Now, for any $n\in\N$ and $\lambda\in\Sigma$, we denote by $u_n$ the unique solution to the equation $\lambda u_n-Q_{1/n,p}u_n=f$ in
$W^{2,p}(\R^N)$. By \eqref{hille-yosida} the sequence
$(u_n)$ is bounded in $L^p(\R^N)$. Since $Q_{1/n,p}u_n=\lambda u_n-f$ also the sequence
$(Q_{1/n,p}u_n)$ is bounded in $L^p(\R^N)$. Hence, by
\eqref{stima-a} and \eqref{stima-b} we can infer that
\begin{equation}
\sup_{n\in\N}\|q_{1/n}^{\frac{1}{2}}|\nabla u_n|\,\|_{L^p(\R^N)}+
\sup_{n\in\N}\|q_{1/n}|D^2u_n|\,\|_{L^p(\R^N)}\le C\|f\|_{L^p(\R^N)},
\label{stima-ve}
\end{equation}
for some positive constant $C$, independent of $f$. Recalling that $q_{1/n}$ is bounded from below by $1/2$, we easily deduce that the
sequence $(u_n)$ is bounded in $W^{2,p}(\R^N)$.
A classical compactness argument shows that, up to a subsequence, $u_n$ converges to some function
$u\in W^{2,p}(\R^N)$, weakly in $W^{2,p}(B_R)$ and strongly in $W^{1,p}(B_R)$, for any $R>0$.
Again, up to a subsequence, we can assume that $u_n$ and $\nabla u_n$ converge, respectively, to $u$ and $\nabla u$,  pointwise in $\R^N$.
Since $\Delta u_n=q_{1/n}^{-1}(\lambda u_n-f+V_{1/n}u_n)$ and $q_{1/n}$ and $V_{1,1/n}$ converge, respectively, to $a$ and $V+1$ locally uniformly in $\R^N$, $\Delta u_n$ converges
in $L^p_{\rm loc}(\R^N)$ to the function $a^{-1}((\lambda+1)u-f+Vu)$. But we already know that, for any $R>0$, $\Delta u_n$ converges weakly in $L^p(B_R)$ to $\Delta u$. Hence, we conclude that
the function $u$ solves the equation $(\lambda+1)u-{\mathcal A}u=f$. Finally, from \eqref{stima-ve} we get
\begin{eqnarray*}
\|a^{\frac{1}{2}}|\nabla u|\,\|_{L^p(\R^N)}+
\|a|D^2u|\,\|_{L^p(\R^N)}\le C\|f\|_{L^p(\R^N)}.
\end{eqnarray*}
By difference,
$Vu=f+a\Delta u-\lambda u$ belongs to $L^p(\R^N)$. Hence, $u\in D(A_p)$.
We have so proved that the equation $\lambda u-A_pu=f$ admits, for any $\lambda\in\C$ with ${\rm Re}\lambda>\omega_0+1$, a solution
$u\in D(A_p)$.
Function $u$ is the unique solution to the equation $\lambda u-A_pu=f$ in $D(A_p)$. Indeed, if $u\in D(A_p)$ solves the equation $\lambda u-A_pu=0$,
then $u\in W^{2,p}(\R^N)$ and $\lambda a^{-1}u-\Delta u+a^{-1}Vu=0$. Multiplying both sides of this equality by $\overline{u}|u|^{p-2}$ and integrating by parts gives
\begin{align*}
0=&\lambda\int_{\R^N}\frac{|u|^p}{a}dx+(p-1)\int_{\R^N}|{\rm Re}(\overline{u}\nabla u)|^2|u|^{p-4}\chi_{\{u\neq 0\}}dx\notag\\
&+\int_{\R^N}|{\rm Im}(\overline{u}\nabla u)|^2|u|^{p-4}\chi_{\{u\neq 0\}}dx\notag\\
&+i(p-2)\int_{\R^N}{\rm Re}(\overline{u}\nabla u)\cdot {\rm Im}(\overline{u}\nabla u)|u|^{p-4}\chi_{\{u\neq 0\}}dx
+\int_{\R^N}a^{-1}V|u|^pdx.
\end{align*}
Taking the real part and recalling that ${\rm Re}\lambda>0$, we conclude that $u\equiv 0$.
Hence $\{\lambda\in\C: {\rm Re}\lambda\ge\omega_0+1\}\subset\rho(A_p)$.

Finally, letting $n$ tend to $\infty$ in \eqref{hille-yosida} gives
\begin{eqnarray*}
|\lambda|\|u\|_{L^p(\R^N)}\le M_p\|\lambda u-A_pu\|_{L^p(\R^N)},
\label{hille-yosida-bis}
\end{eqnarray*}
for any $\alpha\in [0,2]$, any $u\in D(A_p)$ and any $\lambda\in\C$ with ${\rm Re}\lambda\ge\omega_0+1$. By \cite[Prop. 2.1.11]{lunardi}, we conclude
that $A_p$ is a sectorial operator and, therefore, it generates an analytic semigroup $(T_p(t))_{t\ge 0}$ in $L^p(\R^N)$.
Such a semigroup is strongly continuous. Indeed, $D(A_p)$ is dense in $L^p(\R^N)$ since it contains $C^{\infty}_c(\R^N)$.

\emph{Step 4.}
To complete the proof we check that the semigroups $(T_p(t))_{t\ge 0}$ preserve positivity and are all consistent.
In view of the exponential formula
\begin{eqnarray*}
T_p(t)f=\lim_{n\to\infty}\left [\frac{n}{t}R\left (\frac{n}{t},A_p\right )\right ]^nf,
\end{eqnarray*}
which holds for any $t>0$, any $f\in L^p(\R^N)$, where the limit is meant in the norm topology of $L^p(\R^N)$
(see \cite[Chpt. 1, Thm. 8.3]{pazy}), it suffices to prove that the resolvent families
${\mathcal R}_p:=\{R(\lambda,A_p): \lambda>0\}$ ($p\in (1,\infty)$) are consistent and preserve positivity.  The positivity of the resolvent family
${\mathcal R}_p$ for any $p\in (1,\infty)$ follows immediately if we recall that, for any $f\in L^p(\R^N)$, $R(\lambda, A_p)f$
is the limit in $L^p_{\rm loc}(\R^N)$ of the sequence of functions $(R(\lambda,Q_{1/n,p})f)$ and,
by classical results, each operator $R(\lambda,Q_{p,1/n})$
preserves positivity.

Similarly, since, for any $n\in\N$, the resolvent families $\{R(\lambda,Q_{1/n,p}): \lambda>0\}$ ($p\in (1,\infty)$)
are consistent, for any $p,q\in (1,\infty)$, any $n\in\N$ and any $f\in L^p(\R^N)\cap L^q(\R^N)$, the function $R(\lambda,Q_{1/n,p})f$ belongs to $L^p(\R^N)\cap L^q(\R^N)$ and $R(\lambda,A_{1/n,p})f=R(\lambda,A_{1/n,q})f$. Letting $n$ tend to $\infty$ we conclude that $R(\lambda,A_p)f=R(\lambda,A_q)f$.
This concludes the proof.
\end{proof}

\subsection{Some additional properties of the semigroup $(T_p(t))_{t\ge 0}$ and the spectrum of operator $A_p$}

To begin with, we state some remarkable properties of the semigroup $(T_p(t))_{t\ge 0}$.

\begin{proposition}
\label{prop-Cbgamma}
For any $p\in (1,\infty)$, any $f\in L^p(\R^N)$, any $\gamma\in (0,1)$ and any $t>0$, the function $T_p(t)f$ belongs to $C_b^{1+\gamma}(\R^N)$.
In particular, the semigroup $(T_p(t))_{t\ge 0}$ is ultracontractive.
\end{proposition}

\begin{proof}
Fix $f\in L^p(\R^N)$, $\gamma\in (0,1)$ and $t>0$.
Since $(T_p(t))_{t\ge 0}$ is an analytic semigroup in $L^p(\R^N)$, $T_p(t/2)f\in D(A_p)\subset W^{2,p}(\R^N)$. If $p\ge N/2$,
then $T_p(t/2)f\in L^q(\R^N)$ for any $q\in [p,\infty)$. Since the semigroups $(T_r(t))_{t\ge 0}$ $(r\in (1,\infty))$ are consistent, $T_p(t)f=T_q(t/2)T_p(t/2)f\in D(A_q)$.
Choosing $q$ such that $W^{2,q}(\R^N)$ embeds in $C_b^{1+\gamma}(\R^N)$, we conclude that $T_p(t)f\in C_b^{1+\gamma}(\R^N)$.

Let us now suppose that $p< N/2$. Consider the sequence $(r_n)$, defined by $r_n=1/p-2n/N$ for any $n\in\N$, and set $q_n=1/r_n$ for any $n\in\N$. Let $n_0$ be the smallest integer such that $r_{n_0}\le 2/N$; note that $r_{n_0}>0$. Then, $T_p(t/(n_0+2))f\in D(A_p)\subset L^{q_1}(\R^N)\cap L^p(\R^N)$,
by the Sobolev embedding theorem. Hence, $T_p(2t/(n_0+2))f=T_{q_1}(t/(n_0+2))T_p(t/(n_0+2))f\in
D(A_{q_1})\subset L^{q_2}(\R^N)$. Iterating this argument, we obtain that $T_p((n_0+1)t/(n_0+2))f\in D(A_{q_{n_0}})$, and we can conclude that $T_p(t)f\in C_b^{1+\gamma}(\R^N)$ arguing as in the previous case.
The last statement of the proposition is now immediate.
\end{proof}

It is well-known that one can associate a semigroup $(T(t))_{t\ge 0}$ of bounded operators in $C_b(\R^N)$ with operator ${\mathcal A}$.
Since the function $\varphi:\R^N\to\R$, defined by $\varphi(x)=1+|x|^2$ for any $x\in\R^N$, is a Lyapunov function for
the operator ${\mathcal A}$ (i.e., ${\mathcal A}\varphi\le c\varphi$), for any $f\in C_b(\R^N)$, $T(t)f$ is the value at $t>0$ of the unique
solution $u\in C([0,\infty)\times\R^N)\cap C^{1,2}((0,\infty)\times\R^N)$ of the Cauchy problem
\begin{equation}
\left\{
\begin{array}{lll}
D_tu(t,x)={\mathcal A}u(t,x), & t>0, &x\in\R^N,\\[1mm]
u(0,x)=f(x), &&x\in\R^N,
\end{array}
\right.
\label{pb-omogeneo}
\end{equation}
which is bounded in each strip $[0,T]\times\R^N$. Actually, since the potential term in operator ${\mathcal A}$ is nonpositive in $\R^N$, $T(\cdot)f$ is bounded in $[0,\infty)\times\R^N$.
Finally, we stress that $(T(t))_{t\ge 0}$ is strong Feller and irreducible, i.e., each operator $T(t)$ maps $B_b(\R^N)$ (the space of all the bounded and Borel measurable functions $f:\R^N\to\R$) into $C_b(\R^N)$ and
$T(t)\chi_{E}>0$ in $\R^N$ for any measurable set $E$ with positive Lebesgue measure.
We refer the reader to \cite{libro,MPW} for the proofs of the claimed results and for further details.

\begin{proposition}
\label{prop-cons-Linf}
The semigroups $(T_p(t))_{t\ge 0}$ and $(T(t))_{t\ge 0}$ agree on $B_b(\R^N)\cap L^p(\R^N)$ for any $p\in (1,\infty)$.
Moreover, for any $f\in L^p(\R^N)$ and any $t>0$, the function $T_p(t)f$ belongs to $C^{2+\gamma}_{\rm loc}(\R^N)$, where $\gamma=\min\{\alpha,\beta\}$, if $\alpha,\beta >0$, and $\gamma=\max\{\alpha,\beta\}$, if $\alpha\beta=0$.
\end{proposition}

\begin{proof}
Fix $f\in C_c^2(\R^N)\subset D(A_q)$ for any $q\in (1,\infty)$. Since $(T_p(t))_{t\ge 0}$ is a strongly continuous analytic semigroup in $L^p(\R^N)$, for any $p\in (1,\infty)$, and $D(A_p)$ continuously embeds in $W^{2,p}(\R^N)$, the function $v=T_p(\cdot)f$ is in $C^{\infty}((0,\infty);W^{2,p}(\R^N))\cap C([0,\infty);W^{2,p}(\R^N))$ and solves the Cauchy problem
\eqref{pb-omogeneo}. Taking Proposition \ref{prop-Cbgamma} and the Sobolev embedding theorem into account, we can infer that $v$ belongs to $C^{\infty}((0,\infty);C^{1+\theta}_b(\R^N))\cap
C([0,\infty)\times\R^N)$ for any $\theta\in (0,1)$. In particular, $v\in C^{1+\theta}_{\rm loc}((0,\infty)\times\R^N)$ for any $\theta$ as above. By difference, $a\Delta v-Vv\in C^{\theta}_{\rm loc}((0,\infty)\times\R^N)$. As a byproduct, we deduce that $\Delta v\in C^{\theta}_{\rm loc}((0,\infty)\times\R^N)$ and, by elliptic regularity, $D_{ij}v\in C((0,\infty)\times\R^N)$ for any $i,j=1,\ldots,N$. Hence, $v\in C([0,\infty)\times\R^N)\cap C^{1,2}((0,\infty)\times\R^N)$ is a classical solution to problem \eqref{pb-omogeneo} and is bounded in each strip $[0,T]\times\R^N$. By uniqueness, $T_p(\cdot)f=T(\cdot)f$.

Let us now assume that $f\in L^p(\R^N)\cap B_b(\R^N)$. For any $n\in\N$, let us consider the function
$f_n=\vartheta_n(\varrho_n\star f)$, where $(\varrho_n)$ is a standard sequence of mollifiers, $(\vartheta_n)$ is a (standard) sequence
of cut-off functions such that $\chi_{B_n}\le\vartheta_n\le\chi_{B_{2n}}$ for any $n\in\N$.
It is well known that the sequence $(f_n)$ converges to $f$ in $L^p(\R^N)$ and pointwise in $\R^N$ as $n\to\infty$.
Clearly, $T_p(t)f_n$ tends to $T_p(t)f$ in $L^p(\R^N)$, as $n\to\infty$, for any $t>0$. Moreover, since
$\|f_n\|_{\infty}\le\|f\|_{\infty}$ for any $n\in\N$ and $(T(t))_{t\ge 0}$ is strong Feller,
$T(t)f_n$ converges to $T(t)f$ as $n\to\infty$, pointwise in $\R^N$, for any $t>0$ (see e.g., \cite[Cor. 4.7]{MPW}).
Using the semigroup property and \cite[Prop. 4.6]{MPW} (or \cite[Prop. 2.2.9]{libro}) we can infer that $T(t)f_n$ converges to $T(t)f$, locally uniformly in $\R^N$ as $n\to\infty$, for any $t>0$. We thus conclude that $T_p(t)f\equiv T(t)f$ for any $t>0$, and we are done.

To complete the proof we observe that, for any $g\in C_b(\R^N)$ and any $t>0$, the function $T(t)g$ belongs to $C^{2+\gamma}_{\rm loc}(\R^N)$, where $\gamma$ is as in the statement of the theorem. Since $T_p(t/2)f\in C_b(\R^N)$ for any $f\in L^p(\R^N)$, by Proposition \ref{prop-Cbgamma}, we obtain that $T_p(t)f=T(t/2)T_p(t/2)f\in C^{2+\gamma}_{\rm loc}(\R^N)$ as it has been claimed.
\end{proof}

\begin{corollary}\label{cor-irred}
For any $p\in (1,\infty)$ the semigroup $(T_p(t))_{t\ge 0}$ is irreducible, i.e., for any nonnegative and non identically vanishing function $f\in L^p(\R^N)$ and any $t>0$, it holds that $T_p(t)f>0$ in $\R^N$.
\end{corollary}

\begin{proof}
Since the semigroup $(T(t))_{t\ge 0}$ is irreducible, from Proposition
\ref{prop-cons-Linf} we deduce that $T_p(t)\chi_E>0$ in $\R^N$, for any measurable set $E\subset\R^N$ with positive and finite Lebesgue measure,
and any $t>0$.

Let us now fix a nonnegative and non identically vanishing function $f\in L^p(\R^N)$. Then, there exists $m\in\N$ such that
the set $E_m=\{x\in\R^N: f(x)>1/m\}$ has positive Lebesgue measure. Up to intersecting $E_m$ with a sufficiently large ball, we can assume that
the set $E_m$ is bounded and non empty. Since $f\ge m^{-1}\chi_{E_m}$, from the positivity of the semigroup $(T_p(t))_{t\ge 0}$ we conclude that
$T_p(t)f\ge T_p(t)\chi_{E_m}>0$ everywhere in $\R^N$.
\end{proof}

For any $p\in (1,\infty)$, $(A_p,D(A_p))$ is the minimal realization of operator ${\mathcal A}$ in $L^p({\mathbb R}^N)$. The following proposition shows that $(A_p,D(A_p))$ actually coincides with the maximal realization of operator ${\mathcal A}$ in $L^p(\R^N)$.

\begin{proposition}
\label{prop-2.8}
For any $p\in (1,\infty)$ it holds that
\begin{eqnarray*}
D(A_p)=D_{\max,p}({\mathcal A}):=\{u\in L^p({\mathbb R}^N)\cap W^{2,p}_{\rm loc}({\mathbb R}^N):~{\mathcal A} u \in L^p({\mathbb R}^N)\}.
\end{eqnarray*}
\end{proposition}

\begin{proof}
Clearly, we have only to prove the inclusion ``$\supset$''. Fix $p\in (1,\infty)$,
$u\in D_{\max,p}({\mathcal A})$,
$\lambda\in\rho(A_p)\cap \R$ and set
$f:=\lambda u - {\mathcal A} u$. Without loss of generality, we can assume that $u$ is a real-valued function.
The function $v:= u-R(\lambda,A_p)f$
satisfies the equation $\lambda v - {\mathcal A} v = 0$.
We shall show that $v\equiv 0$, provided
$\lambda$ is large enough.

We first consider the case $\alpha\in [1,2$].
Integrating the identity $(\lambda v - {\mathcal A} v)v|v|^{p-2}\vartheta_n^2 = 0$ by parts on ${\mathbb R}^N$,
where $(\vartheta_n)$ is a standard sequence of cut-off functions, we get
\begin{align}
0=&\lambda \int_{{\mathbb R}^N}|v|^p\vartheta_n^2dx
+\int_{{\mathbb R}^N}V|v|^p\vartheta_n^2dx+ (p-1)\int_{{\mathbb R}^N}a|\nabla v|^2|v|^{p-2}\vartheta_n^2\chi_{\{v\neq 0\}}dx
\nonumber\\[1mm]
&+2\int_{{\mathbb R}^N}a\vartheta_n|v|^{p-2}v\nabla v\cdot\nabla\vartheta_n\,dx
+\int_{{\mathbb R}^N}\vartheta_n^2|v|^{p-2}v\nabla a\cdot\nabla v\,dx.
\label{mart-1}
\end{align}

Note that $C_1:=\sup_{n\in\N}\|a^{1/2}|\nabla\vartheta_n|\,\|_{\infty}<\infty$. Hence,
\begin{align}
&\left |\int_{{\mathbb R}^N}a\vartheta_n|v|^{p-2}v\nabla v\cdot\nabla\vartheta_n\,dx\right |\notag\\
\le &C_1\int_{{\mathbb R}^N\setminus B_n}a^{\frac{1}{2}}|v|^{p-1}|\nabla v|\vartheta_ndx\nonumber\\
\le & C_1\bigg (\int_{{\mathbb R}^N} a|\nabla v|^2|v|^{p-2}\vartheta_n^2\chi_{\{v\neq 0\}}dx\bigg )^{\frac{1}{2}}
\bigg (\int_{\R^N\setminus B_n}|v|^pdx\bigg )^{\frac{1}{2}}\nonumber\\[1mm]
\le&\varepsilon\int_{{\mathbb R}^N}a|\nabla
v|^2|v|^{p-2}\vartheta_n^2\chi_{\{v\neq 0\}}dx +\frac{C_1^2}{4\varepsilon}\int_{\R^N\setminus B_n}|v|^pdx
 \label{mart-2}
\end{align}
and, since $|\nabla a|\le \alpha a^{1/2}$, we can estimate (in a completely similar way)
\begin{align}
\left |\int_{{\mathbb R}^N}\vartheta_n^2|v|^{p-2}v\nabla a\cdot\nabla v\,dx\right |
\le\varepsilon\int_{{\mathbb R}^N}a|\nabla v|^2|v|^{p-2}\vartheta_n^2\chi_{\{v\neq 0\}}dx+\frac{\alpha^2}{4\varepsilon}\int_{{\mathbb R}^N}|v|^p\vartheta_n^2dx,
\label{mart-3}
\end{align}
for any $\varepsilon>0$.
Replacing \eqref{mart-2} and \eqref{mart-3} into \eqref{mart-1} and taking $\varepsilon=(p-1)/3$, gives
\begin{align}
\left (\lambda-\frac{3\alpha^2}{4(p-1)}\right )\int_{{\mathbb R}^N}|v|^p\vartheta_n^2dx
-\frac{3C_1^2}{2(p-1)}\int_{\R^N\setminus B_n}|v|^pdx\le 0.
\label{mart-4}
\end{align}
Letting $n$ tend to $\infty$ yields $v\equiv 0$, if we take $\lambda$ large enough.

Let us now assume that $\alpha\in [0,1)$ and consider the function $\hat a$
defined by $\hat a=\varphi\tilde a+(1-\varphi)a$, where $\tilde a$ is any
smooth function such that $\tilde a\ge 1$ and $\|a-\tilde a\|_{L^{\infty}(B_2)}\le 1/2$, and $\varphi\in C^{\infty}_c(\R^N)$ satisfies
$\chi_{B_1}\le\varphi\le\chi_{B_2}$. Function $\tilde a$ can be obtained, for instance, regularizing by convolution $a$.
Let $\hat{\mathcal A}$ be the operator defined as $\mathcal{A}$ with $a$ being replaced by $\hat a$.
As in the case when $\alpha\ge 1$, we multiply the equation $\lambda v-{\mathcal A}v=0$ by $v|v|^{p-2}\vartheta_n^2$. We get
\begin{align}
0=&\int_{\R^N}(\lambda v-{\mathcal A}v)v|v|^{p-2}\vartheta_n^2dx\notag\\
=&\int_{\R^N}(\lambda v-\hat{\mathcal A}v)v|v|^{p-2}\vartheta_n^2dx
+\int_{B_2}(\hat a-a)v|v|^{p-2}\vartheta_n^2\Delta vdx.
\label{max-1}
\end{align}
Since $\hat a$ is continuously differentiable in $\R^N$ and $|\nabla\hat a|\le\hat\kappa\hat a^{1/2}$ for some positive constant $\hat\kappa$, we can integrate by parts the first term in the last side of \eqref{max-1}. Arguing as in the proof of \eqref{mart-4} we can estimate
\begin{align}
\int_{\R^N}(\lambda v-\hat{\mathcal A}v)v|v|^{p-2}\vartheta_n^2dx\ge&\left (\lambda-\frac{3\hat\kappa^2}{4(p-1)}\right )\int_{\R^N}|v|^p\vartheta_n^2dx\notag\\
&-\frac{3\hat C_1^2}{2(p-1)}\int_{\R^N\setminus B_n}|v|^pdx,
\label{max-2}
\end{align}
where $\hat C_1=\sup_{n\in\N}\|\hat a^{1/2}|\nabla\vartheta_n|\,\|_{\infty}$.
As far as the other term in the last side of \eqref{max-1} is concerned, we observe that, since $a\ge 1$ in $\R^N$, it holds that
\begin{align}
\left |\int_{B_2}(\hat a-a)v|v|^{p-2}\vartheta_n^2\Delta vdx\right |
\le &\|a-\hat a\|_{L^{\infty}(B_2)}\int_{B_2}|v|^{p-1}|a\Delta v|dx\notag\\
\le & \|a-\tilde a\|_{L^{\infty}(B_2)}\|v\|_{L^p(B_2)}^{p-1}\|a\Delta v\|_{L^p(B_2)}\notag\\
=& \frac{1}{2}\|v\|_{L^p(B_2)}^{p-1}\|(V+\lambda)v\|_{L^p(B_2)}\notag\\
\le &\frac{1}{2}\left (2^{\beta}+\lambda\right )\|v\|_{L^p(\R^N)}^p.
\label{max-3}
\end{align}
From \eqref{max-1}, \eqref{max-2} and \eqref{max-3} we obtain
\begin{align*}
0\ge&\left (\lambda-\frac{3\hat\kappa^2}{4(p-1)}\right )\int_{\R^N}|v|^p\vartheta_n^2dx
-\frac{3\hat C_1^2}{2(p-1)}\int_{\R^N\setminus B_n}|v|^pdx\\
&-\frac{1}{2}(2^{\beta}+\lambda)\int_{\R^N}|v|^pdx.
\end{align*}
Again, letting $n\to\infty$ we get
\begin{eqnarray*}
\left (\frac{\lambda}{2}-\frac{3\hat\kappa^2}{4(p-1)}-2^{\beta-1}\right )\int_{\R^N}|v|^pdx\le 0,
\end{eqnarray*}
and, then, choosing $\lambda$ large enough, we conclude that $v\equiv 0$.
\end{proof}

\begin{proposition}
\label{prop-2.13}
For any $\alpha\in [0,2]$, any $\beta>0$ and any $p\in (1,\infty)$ the spectrum of $A_p$
consists of a sequence of negative real eigenvalues which accumulates at $-\infty$. Moreover, $\sigma(A_p)$ is independent of $p$.
\end{proposition}

\begin{proof}
The proof is split into three steps. In the first one we prove that, for any $p\in (1,\infty)$, $\sigma(A_p)$ consists of isolated eigenvalues. Then, we prove that $\sigma(A_p)$ is independent of $p$ and, finally, we show that the eigenvalues of $A_p$ are real and negative.

{\em Step 1.} Fix $p\in (1,\infty)$. To prove that the spectrum of $A_p$
consists of eigenvalues only, let us show that $D(A_p)$ is compactly embedded into $L^p(\R^N)$ for any $p\in (1,\infty)$.
This will yield immediately that the resolvent operator $R(\lambda,A_p)$ is compact in $L^p(\R^N)$ for any $\lambda\in\rho(A_p)$.
Hence, its spectrum (and, consequently, the spectrum of $A_p$) consists of eigenvalues.

Since
\begin{eqnarray*}
\int_{\R^N}|x|^{\beta p}|u(x)|^pdx\le C_1\left (\int_{\R^N}|u(x)|^pdx+\int_{\R^N}|A_pu(x)|^pdx\right ),
\end{eqnarray*}
for some positive constant $C_1$, independent of $u$, taking Corollary \ref{cor-2.6} into account we can easily conclude that there exists a positive constant $C_2$, independent of $u$ as well, such that
\begin{equation}
\int_{\R^N}|x|^{\beta p}|u(x)|^pdx\le C_2,
\label{stima-per-comp}
\end{equation}
for any $u\in {\mathcal B}:=\{v\in D(A_p): \|v\|_{D(A_p)}\le 1\}$.
This estimate yields the compactness of ${\mathcal B}$ in $L^p(\R^N)$ by a standard argument. Anyway for the reader's convenience we give some details.
To prove that ${\mathcal B}$ is compact in $L^p(\R^N)$, we show that it is totally bounded.
By estimate \eqref{stima-per-comp} we deduce that
\begin{equation}
\int_{\R^N\setminus B_M}|u(x)|^pdx\le M^{-\beta p}\int_{\R^N\setminus B_M}|x|^{\beta p}|u(x)|^pdx\le C_2M^{-\beta p},\qquad\;\,u\in {\mathcal B}.
\label{stima-per-comp-2}
\end{equation}
Let us fix $\varepsilon>0$ and let $M_{\varepsilon}$ be large enough such that
\begin{eqnarray*}
\int_{\R^N\setminus B_{M_{\varepsilon}}}|u(x)|^pdx\le \frac{1}{2}\varepsilon^p,\qquad\;\,u\in {\mathcal B}.
\end{eqnarray*}
Since $D(A_p)$ is continuously embedded into $W^{2,p}(\R^N)$, the set ${\mathcal B}_{|B_{M_{\varepsilon}}}$ of the restrictions  to $B_{M_{\varepsilon}}$ of all the functions in ${\mathcal B}$ is continuously embedded in $W^{2,p}(B_{M_{\varepsilon}})$. As this latter space is compactly embedded in $L^p(B_{M_{\varepsilon}})$,
there exist $n_{\varepsilon}\in\N$ and functions $f_1,\ldots,f_{n_{\varepsilon}}$ in $L^p(B_{M_{\varepsilon}})$ such that,
for any $u\in {\mathcal B}$ and some $j=j(u)\in\{1,\ldots,n_{\varepsilon}\}$,
\begin{equation}
\int_{B_{M_{\varepsilon}}}|u(x)-f_j(x)|^pdx\le\frac{1}{2}\varepsilon^p.
\label{stima-per-comp-3}
\end{equation}
Let us now denote by $\tilde f_j$ ($j=1,\ldots,n_{\varepsilon}$) the function which equals $f_j$ in $B_{M_{\varepsilon}}$ and identically vanishes elsewhere in $\R^N$. Using \eqref{stima-per-comp-2} and \eqref{stima-per-comp-3} we obtain that $\|u-\tilde f_j\|_{L^p(\R^N)}\le\varepsilon$, and this shows that ${\mathcal B}$ is totally bounded in $L^p(\R^N)$.

{\em Step 2.} Let us now show that the spectrum of $A_p$ is independent of $p\in (1,\infty)$. The proof that we present is obtained adapting the arguments in the proof of \cite[Cor. 1.6.2]{davies}.
Fix $p,q\in (1,\infty)$ and $f\in C^{\infty}_c(\R^N)$. By the proof of Theorem \ref{prop-2} we know that
 the operators $A_p$ and $A_q$ are sectorial. Hence, we can determine $\omega>0$ such that
the interval $(\omega,\infty)$ is contained in both the resolvent sets of operators $A_p$ and $A_q$.

Since the semigroups $(T_p(t))_{t\ge 0}$ and $(T_q(t))_{t\ge 0}$ coincide on
$L^p(\R^N)\cap L^q(\R^N)$, $T_p(t)f=T_q(t)f$ for any $t>0$. Therefore, for $\lambda>\omega$ we get
\begin{eqnarray*}
R(\lambda,A_p)f=\int_0^{\infty}e^{-\lambda t}T_p(t)fdt=\int_0^{\infty}e^{-\lambda t}T_q(t)fdt=R(\lambda,A_q)f.
\end{eqnarray*}
In particular,
\begin{eqnarray*}
\int_{\R^N}gR(\lambda,A_p)fdx=\int_{\R^N}gR(\lambda,A_q)fdx,\qquad\;\,\lambda>\omega,
\end{eqnarray*}
for any $g\in C^{\infty}_c(\R^N)$.
Note that, since $\sigma(A_p)$ and $\sigma(A_q)$ consist of isolated eigenvalues, $\C\setminus (\sigma(A_p)\cup\sigma(A_q))$ is a connected open set in $\C$. Hence, the previous equality can be extended to any $\lambda\in\C\setminus (\sigma(A_p)\cup\sigma(A_q))$.
The arbitrariness of $g$ shows that $R(\lambda,A_p)f=R(\lambda,A_q)f$ for any $\lambda\in\C\setminus (\sigma(A_p)\cup\sigma(A_q))$.
Let us now fix $\lambda_0\in \sigma(A_p)$. Since both $\sigma(A_p)$ and $\sigma(A_q)$ admit no accumulation points in $\mathbb C$,
$\lambda_0$ is isolated in $\sigma(A_p)\cup\sigma(A_q)$. Hence, we can determine $\varepsilon>0$ small enough such that
$B_{\varepsilon}(\lambda_0)\setminus\{\lambda_0\}\subset \C\setminus (\sigma(A_p)\cup\sigma(A_q))$.
Let $P$ be the spectral projection associated to the eigenvalue $\lambda_0\in\sigma(A_p)$, which is defined by
\begin{eqnarray*}
Ph=\frac{1}{2\pi i}\int_{\partial B_{\varepsilon}(\lambda_0)}R(\lambda,A_p)hd\lambda,\qquad\;\,h\in L^p(\R^N),
\end{eqnarray*}
where $\partial B_{\varepsilon}(\lambda_0)$ is oriented counterclockwise. If $\lambda_0\notin\sigma(A_q)$, from the above arguments
we obtain that
\begin{eqnarray*}
Pf=\frac{1}{2\pi i}\int_{\partial B_{\varepsilon}(\lambda_0)}R(\lambda,A_q)fd\lambda=0,
\end{eqnarray*}
which implies that $P\equiv 0$ by density: a contradiction.
Hence, $\sigma(A_p)\subset\sigma(A_q)$. Since $p$ and $q$ have been arbitrarily fixed, $\sigma(A_q)=\sigma(A_p)$.

{\em Step 3.} We now prove that, for any $p\in (1,\infty)$, the spectrum of $A_p$ consists of negative eigenvalues. In view of Step 2, we can limit ourselves to dealing with the case $p=2$. Let $\lambda\in\sigma(A_2)$ and $u\in D(A_2)$ be such that
$\lambda u-A_2u=0$. Multiplying both sides of this equality by $a^{-1}\overline{u}$ and integrating by parts we
get
\begin{align*}
0=&\lambda\int_{\R^N}\frac{|u|^2}{a}dx-\int_{\R^N}\overline{u}\Delta udx
+\int_{\R^N}V\frac{|u|^p}{a}dx\\
=&\lambda\int_{\R^N}\frac{|u|^2}{a}dx+\int_{\R^N}|\nabla u|^2dx
+\int_{\R^N}V\frac{|u|^2}{a}dx.
\end{align*}
From the first and last side of this chain of equalities we immediately infer that $\lambda$ is real and negative.

Finally, we observe that the eigenvalues of operator $A_p$ can be ordered into a sequence diverging to $-\infty$. Indeed, if they were a finite number, the operator $A_p$ would be bounded in $L^p(\R^N)$, which, of course, cannot be the case.
This concludes the proof.
\end{proof}

To conclude this subsection we prove the following result, which will be used in Section \ref{sect-3}.

\begin{proposition}
\label{thm-eigenfunct}
Suppose that $\alpha\in [0,2]$ and $\beta>0$. Then, for any $p\in (1,\infty)$, the eigenspace corresponding to the largest eigenvalue $\lambda_0$ of $A_p$ is one-dimensional and is spanned by a strictly positive function $\phi$, which is radial, belongs to $C_b^{1+\gamma}(\R^N)\cap C^2(\R^N)$ for any $\gamma\in (0,1)$ and tends to $0$ as $|x|\to\infty$.
\end{proposition}

\begin{proof}
Fix $p\in (1,\infty)$. By Proposition \ref{prop-2.13} we know that the spectrum of $A_p$ consists of a sequence of negative isolated eigenvalues with finite
geometric multiplicity. It follows from \cite[A-III, Prop. 2.5(iii)]{nagel86} and the Riesz-Schauder theory for compact operators (cf. \cite[VII.4.5]{DS58}) that $\lambda_0$ is a pole of $R(\cdot ,A_p)$ of finite algebraic multiplicity.
From Corollary \ref{cor-irred} and \cite[C-III, Prop. 3.5]{nagel86} we now
conclude that the eigenspace corresponding to $\lambda_0$ is one-dimensional and is spanned by a strictly positive function $\phi$.
Since $e^{\lambda_0}\phi=T_p(1)\phi$, from Propositions \ref{prop-Cbgamma} and \ref{prop-cons-Linf} it follows immediately that $\phi\in C^{1+\gamma}_b(\R^N)\cap C^2(\R^N)$ for any $\gamma\in (0,1)$. In particular, $\phi=e^{-\lambda_0}T_N(1)\phi$. Hence,
$\phi(x)$ vanishes as $|x|\to\infty$.

To complete the proof let us prove that $\phi$ is radial. For this purpose, we observe that, since the coefficients of operator ${\mathcal A}$ are radial, the function $x\mapsto\phi_R(x):=\phi(Rx)$ is an eigenfunction of $\mathcal{A}$ associated to the eigenvalue $\lambda_0$, for any orthogonal matrix $R$. Hence, $\phi_R$ and $\phi$ should be proportional. Since they coincide at $x=0$, they should coincide everywhere in $\R^N$. Hence, $\phi(Rx)=\phi(x)$ for any $x\in\R^N$ and this shows that
$\phi$ is radial.
\end{proof}

\begin{remark}
\label{cor-2.16}
From Proposition \ref{thm-eigenfunct} we know that the largest eigenvalue $\lambda_0$ of $A_p$ is simple. So, it follows from \cite[Cor. 2.3.5]{lunardi} that
there exists a positive constant $M_p$ such that
\begin{equation}
\|T_p(t)\|_{L(L^p(\R^N))}\le M_pe^{\lambda_0t},\qquad\;\,t\ge 0.
\label{spectr-behav}
\end{equation}
\end{remark}

\section{Heat kernel estimates}\label{sect-3}
It follows from the local regularity of the coefficients of the operator $\mathcal{A}$ that the semigroup $(T_p(t))_{t\ge 0}$ generated by $A_p$ admits a heat kernel $k(t,x,y)$ such that
\begin{eqnarray*}
T_p(t)f(x)=\int_{\R^N}k(t,x,y)f(y)\,dy,\qquad\;\, t>0,\;\,x\in \R^N,
\end{eqnarray*}
for all $f\in L^p(\R^N)$ (cf. \cite{MPW}).

In this section we propose to prove some upper estimates for $k$. For this purpose, let us
estimate the eigenfunction $\phi$ corresponding to the largest eigenvalue $\lambda_0$ of the operator $\mathcal{A}$.
\begin{proposition}\label{estim-eigenfunct}
Assume that $\beta >0$ and $\alpha \in [0,2)$. Then, there exist constants $C_1,\,C_2>0$ such that
\begin{eqnarray*}
C_1f_0(x)\le \phi(x)\le C_2f_{2\lambda_0}(x),
\end{eqnarray*}
for all $x\in\R^N\setminus B(0,1)$, where
\begin{equation}
f_\lambda(x):=|x|^{\frac{\alpha-\beta}{4}-\frac{N-1}{2}}\exp\left(-\int_1^{|x|}\frac{s^{\beta/2}}{(1+s^\alpha)^{1/2}}\,ds
-\frac{\lambda}{2}\int_1^{|x|}\frac{1}{s^{\beta/2}(1+s^\alpha)^{1/2}}\,ds\right).
\label{flambda}
\end{equation}
\end{proposition}
\begin{proof}
Let us set $f_{\lambda}(x)=\psi(|x|)e^{-g_{\lambda}(|x|)}$, where $\lambda$ is a real constant,
\begin{eqnarray*}
g'_{\lambda}(r)=\frac{r^{\beta/2}}{(1+r^{\alpha})^{1/2}}+\frac{\lambda}{2r^{\beta/2}(1+r^{\alpha})^{1/2}},
\qquad\;\,r>0,
\end{eqnarray*}
and $\psi$ is a positive and sufficiently smooth function to be chosen later on.

A straightforward computation shows that
\begin{align*}
\Delta f_{\lambda}(x)
=&\frac{\psi''(|x|)}{\psi(|x|)}f_{\lambda}(x)-\left (g_{\lambda}''(|x|)+2g_{\lambda}'(|x|)\frac{\psi'(|x|)}{\psi(|x|)}
+\frac{N-1}{|x|}g_{\lambda}'(|x|)\right )f_{\lambda}(x)
\notag\\
&+(g_{\lambda}'(|x|))^2f_{\lambda}(x)+\frac{N-1}{|x|}\cdot\frac{\psi'(|x|)}{\psi(|x|)}f_{\lambda}(x),
\end{align*}
for any $x\in\R^N\setminus\{0\}$.
Let us determine $\psi$ in such a way that
\begin{align*}
g_0''(r)+2g'_0(r)\frac{\psi'(r)}{\psi(r)}+\frac{N-1}{r}g_0'(r)=0,\qquad\;\,r>0.
\end{align*}
We can take
\begin{eqnarray*}
\psi(r)=r^{-\frac{N-1}{2}}\left (\frac{1+r^{\alpha}}{r^{\beta}}\right )^{\frac{1}{4}},\qquad\;\,r>0.
\end{eqnarray*}
With this choice of $\psi$ we get
\begin{align*}
{\mathcal A}f_{\lambda}(x)-\lambda f_{\lambda}(x)
=&(1+|x|^{\alpha})\frac{\psi''(|x|)}{\psi(|x|)}f_{\lambda}(x)
-\lambda\frac{(1+|x|^{\alpha})^{1/2}}{|x|^{\beta/2}}\cdot\frac{\psi'(|x|)}{\psi(|x|)}f_{\lambda}(x)\notag\\
&-\frac{\lambda}{2}(1+|x|^{\alpha})
\left (\frac{d}{dr}\frac{1}{r^{\beta/2}(1+r^{\alpha})^{1/2}}\right )_{|r=|x|}f_{\lambda}(x)\notag\\
&+\frac{\lambda^2}{4|x|^{\beta}}f_{\lambda}(|x|)+\frac{N-1}{|x|}(1+|x|^{\alpha})\frac{\psi'(|x|)}{\psi(|x|)}f_{\lambda}(x)\notag\\
&-\lambda\frac{N-1}{2|x|}\cdot\frac{(1+|x|^{\alpha})^{1/2}}{|x|^{\beta/2}}f_{\lambda}(x)\notag\\
=:&h_{\lambda}(|x|)f_{\lambda}(x),
\end{align*}
for any $x\in\R^N\setminus\{0\}$.
Since $\alpha\in [0,2)$ and $\beta>0$, $h_{\lambda}(|x|)$ tends to $0$ as $|x|\to\infty$.
Therefore,
\begin{equation}
{\mathcal A}f_{\lambda}(x)-\lambda f_{\lambda}(x)=o(1)f_{\lambda}(x),
\label{star*}
\end{equation}
as $|x|\to \infty$. We can thus apply the same arguments as in Davies book \cite{davies}. More precisely, if $\phi$ is a positive eigenfunction of operator $A_p$ associated to the largest
eigenvalue $\lambda_0<0$, then we have ${\mathcal A}\phi-\lambda_0\phi=0$ in $\R^N$, $\phi\in C_b(\R^N)\cap C^2(\R^N)$ and it
 vanishes as $|x|\to\infty$ (see Proposition \ref{thm-eigenfunct}).
Now, $o(1)>\lambda_0$
if $|x|$ is sufficiently large (let us say $|x|\ge R$). Since $f_0>0$, it holds that
\begin{align*}
0={\mathcal A}f_0(x)-o(1)f_0(x)\le {\mathcal A}f_0(x)-\lambda_0f_0(x),
\end{align*}
for any $x\in\R^N\setminus B_R$, which yields
${\mathcal A}(f_0-\phi)-\lambda_0(f_0-\phi)\ge 0$ in $\R^N\setminus B_R$. Up to replacing $R$ with a larger value if needed, we
can assume that $|x|^{\beta}+\lambda_0>0$ is positive for any $x\in\R^N\setminus B_R$.
Since both $f_0$ and $\phi$ tend to $0$ as $|x|\to\infty$,
 and $f_0^{-1}\phi \ge C_1$ on $\partial B_R$, by compactness, we get, by the maximum principle,
$C_1f_0-\phi\le 0$ in $\R^N\setminus B_R$.

Analogously, by \eqref{star*} it follows that ${\mathcal A}f_{2\lambda_0}-(2\lambda_0+o(1))f_{2\lambda_0}=0$
in $\R^N\setminus\{0\}$. Now, up to replacing $R$ with a larger value, we can assume that
\begin{eqnarray*}
0<|x|^{\beta}+2\lambda_0+o(1)<|x|^{\beta}+\lambda_0,\qquad\;\,x\in\R^N\setminus B_R.
\end{eqnarray*}
Hence, ${\mathcal A}f_{2\lambda_0}-\lambda_0f_{2\lambda_0}\le 0$ in $\R^N\setminus B_R$.
This implies that
${\mathcal A}(f_{2\lambda_0}-\phi)-\lambda_0(f_{2\lambda_0}-\phi)\le 0$ in $\R^N\setminus B_R$ and
as above we can conclude that
$C_2f_{2\lambda_0}-\phi\ge 0$ in $\R^N\setminus B_R$
for some constant $C_2>0$.
\end{proof}

Note that the function $f_{2\lambda_0}/f_0$ is bounded in a neighborhood of $\infty$ if and only $\beta+\alpha>2$. In such a situation one obtains the following result which can be deduced also from \cite[Chpt. 6, Thm. 2.1]{olver}.
\begin{corollary}\label{estim-eigenfunct0}
If $2<\alpha +\beta <2+\beta$ then there exist constants $C_1,\,C_2>0$ such that
\begin{eqnarray*}
C_1f_0(x)\le \phi(x)\le C_2f_0(x),\qquad\;\,x\in\R^N\setminus B(0,1).
\end{eqnarray*}
\end{corollary}

Let us now introduce on $L^2_\mu$ the bilinear form
\begin{eqnarray*}
a_\mu(u,v)=\int_{\R^N}\nabla u\cdot \nabla \overline{v}\,dx+\int_{\R^N}Vu\overline{v}\,d\mu,\qquad\;\, u,v\in D(a_\mu),
\end{eqnarray*}
where $d\mu(x)=(1+|x|^\alpha)^{-1}dx$ and $D(a_\mu)=\{u\in L^2_\mu: V^{1/2}u\in L^2_{\mu},\, \nabla u\in (L^2(\R^N))^N\}$.
$D(a_{\mu})$ is a Hilbert space when endowed with the inner product
\begin{eqnarray*}
\langle u,v\rangle_{D(a_{\mu})}=\int_{\R^N}(1+V)u\overline{v}\,d\mu +\int_{\R^N}\nabla u\cdot \nabla \overline{v}\,dx.
\end{eqnarray*}
Since $a_{\mu}$ is a closed, symmetric and accretive form, to $a_\mu$ we associate the self-adjoint operator $A_\mu$ defined by
\begin{eqnarray*}
\left\{
\begin{array}{l}
D(A_\mu)=\displaystyle\left\{u\in L^2_\mu: \exists g\in L^2_\mu \mbox{ such that } a_\mu(u,v)=-\int_{\R^N}g\overline{v}\,d\mu,\,\forall v\in D(a_\mu)\right\},\\[3mm]
A_\mu u=g,
\end{array}
\right.
\end{eqnarray*}
see \cite[Prop. 1.24]{ouhabaz}.
By general results on positive self-adjoint operators induced by nonnegative quadratic forms
in Hilbert spaces (see e.g., \cite[Prop. 1.51, Thms. 1.52, 2.6, 2.13]{ouhabaz})
$A_\mu$ generates a positive analytic semigroup $e^{tA_\mu}$ in $L^2_\mu$.
We denote by $k_{\mu}$ the heat kernel associated to $A_{\mu}$, i.e.,
\begin{eqnarray*}
e^{tA_{\mu}}f=\int_{\R^N}k_{\mu}(t,\cdot,y)f(y)d\mu(y),\qquad\;\,t>0,\;\,f\in L^2_{\mu}.
\end{eqnarray*}

\begin{lemma}
\label{lemma-3.4}
The following properties are satisfied.
\begin{enumerate}[\rm (i)]
\item
$D(A_\mu) = \{u\in D(a_\mu)\cap W^{2,2}_{\rm loc}(\R^N): \mathcal{A}u\in L^2_\mu \}$ and $A_{\mu}u={\mathcal A}u$ for any
$u\in D(A_{\mu})$;
\item
$e^{tA_\mu}f=T_p(t)f$ for any $t\ge 0$, any $p\in (1,\infty)$ and any $f\in L^p(\R^N)\cap L^2_\mu$.
\end{enumerate}
As a byproduct, it follows that
\begin{equation}\label{k-mu}
k_\mu(t,x,y)=(1+|y|^\alpha)k(t,x,y),\qquad\;\, t>0,\;\,x,y\in \R^N.
\end{equation}
\end{lemma}

\begin{proof}
(i) We begin by proving the inclusion ``$\subset$''. Fix $u\in D(A_{\mu})\subset W^{1,2}_{\rm loc}(\R^N)$. Then,
\begin{align*}
\int_{\R^N}\nabla u\cdot\nabla\overline{v}dx=&-\int_{\R^N}\left (A_{\mu}u+Vu\right )\overline{v}d\mu,
\end{align*}
for any $v\in C^{\infty}_c(\R^N)$. Since $g=A_{\mu}u+Vu\in L^2_{\rm loc}(\mathbb R^N)$, from the previous formula we deduce that
\begin{eqnarray*}
\bigg |\int_{\mathbb R^N}\nabla u\cdot\nabla\overline{v}\,dx\bigg |\le \|g\|_{L^2(B_R)}\|v\|_{L^2(\mathbb R^N)},
\end{eqnarray*}
for any $v\in C^{\infty}_c(B_R)$ and any $R>0$. By density this estimate can be extended to any $v\in W^{1,2}_0(B_R)$
for any $R>0$ and, using a standard argument, it is immediate to check that
$\nabla u\in (W^{1,2}_{\rm loc}(\R^N))^N$.
Hence, $u\in W^{2,2}_{\rm loc}(\R^N)$.

Finally, integrating by parts we conclude that $A_{\mu}u={\mathcal A}u$. The inclusion ``$\subset$'' follows at once.

Let us now prove the inclusion ``$\supset$''. For this purpose, we fix $u\in D(a_{\mu})\cap W^{2,2}_{\rm loc}(\R^N)$ such that
$f:={\mathcal A}u\in L^2_{\mu}$, and $\varphi\in C^1_c(\R^N)$. Integrating by parts we get
\begin{align*}
\int_{\R^N}f\overline{\varphi} d\mu=&\int_{\R^N}(\Delta u-a^{-1}Vu)\overline{\varphi} dx\notag\\
=&-\int_{\R^N}\nabla u\cdot\nabla\overline{\varphi} dx-\int_{\R^N}Vu\overline{\varphi} d\mu\notag\\
=&-a_{\mu}(u,\varphi).
\end{align*}
To conclude that $u\in D(A_{\mu})$ we need to show that the previous equality can be extended to any $\varphi\in D(a_{\mu})$. But this follows immediately from the density of $C^1_c(\R^N)$ in $D(a_{\mu})$, which can be proved arguing as in the proof of Proposition
\ref{lemma-1}.
%
The inclusion ``$\supset$'' follows.

(ii) Since $L^2(\R^N)\subset L^2_{\mu}$, with a continuous embedding, by Proposition \ref{prop-2.8} it follows that
$D(A_2)\subset D(A_{\mu})$ with a continuous embedding (when the two previous spaces are endowed with the graph norms).
Hence, for any $f\in L^2(\R^N)$, both the functions $t\mapsto T_2(t)f$ and $t\mapsto e^{tA_{\mu}}f$
belong to $C^1((0,+\infty);L^2_{\mu})\cap C([0,+\infty);L^2_{\mu})\cap C((0,+\infty);D(A_{\mu}))$ and
solve the Cauchy problem
\begin{eqnarray*}
\left\{
\begin{array}{ll}
v'(t)={\mathcal A}v(t), &t\ge 0,\\[1mm]
v(0)=f.
\end{array}
\right.
\end{eqnarray*}
Since the previous problem admits a unique solution with the claimed regularity properties, $T_2(t)f=e^{tA_{\mu}}f$ for
any $t>0$. Recalling that $T_2(t)$ and $T_p(t)$ agree on $L^p(\R^N)\cap L^2(\R^N)$ for
any $t>0$, we conclude that, for any $t>0$, $e^{tA_{\mu}}$ and $T_p(t)$ coincide on $L^p(\R^N)\cap L^2_{\mu}$.
Formula \eqref{k-mu} now follows immediately.
\end{proof}

Let us now give the first application of Proposition \ref{estim-eigenfunct}.
\begin{proposition}\label{on-diagonal}
If $\alpha \in [0,2)$ and $\beta >0$, then
\begin{eqnarray*}
k(t,x,x)\ge Me^{\lambda_0 t}(f_0(x))^2(1+|x|^\alpha)^{-1},\qquad\;\, t>0,
\end{eqnarray*}
for all $x\in\R^N\setminus B(0,1)$ and some constant $M>0$. Here, $f_0$ is given by \eqref{flambda}.
\end{proposition}
\begin{proof}
From the semigroup law and the symmetry of $k_\mu(t,\cdot,\cdot)$ for any $t>0$, we deduce that
\begin{equation}
k_\mu(t,x,x)=\int_{\R^N}k_\mu(t/2,x,y)^2\,d\mu(y),\qquad\;\, t>0,\;\,x\in \R^N.
\label{kmu-0}
\end{equation}
Indeed, for any real valued functions $\varphi,\psi\in C^{\infty}_c(\R^N)$ it holds that
\begin{align*}
\int_{\R^N}\psi e^{tA_{\mu}}\varphi d\mu=&\int_{\R^N}\psi e^{\frac{t}{2}A_{\mu}}e^{\frac{t}{2}A_{\mu}}\varphi d\mu
=\int_{\R^N}e^{\frac{t}{2}A_{\mu}}\psi e^{\frac{t}{2}A_{\mu}}\varphi d\mu
\end{align*}
or, equivalently,
\begin{align*}
&\int_{\R^N}\psi(x)d\mu(x)\int_{\R^N}k_{\mu}(t,x,y)\varphi(y)d\mu(y)\\
=&
\int_{\R^N}d\mu(x)\int_{\R^N}\psi(y)d\mu(y)\int_{\R^N}k_{\mu}(t/2,x,y)k_{\mu}(t/2,x,z)\varphi(z)d\mu(z)\\
=&
\int_{\R^N}\psi(y)d\mu(y)\int_{\R^N}\varphi(z) d\mu(z)\int_{\R^N}k_{\mu}(t/2,x,y)k_{\mu}(t/2,x,z)d\mu(x).
\end{align*}
The arbitrariness of $\varphi,\psi\in C^{\infty}_c(\R^N)$ imply that
\begin{eqnarray*}
k_{\mu}(t,x,y)=\int_{\R^N}k_{\mu}(t/2,z,x)k_{\mu}(t/2,z,y)d\mu(z),
\end{eqnarray*}
and, then, the symmetry of $k_{\mu}(t,\cdot,\cdot)$ leads to \eqref{kmu-0}.

Let us denote by $\phi$ the normalized eigenfunction of ${\mathcal A}$ (i.e., $\|\phi \|_{L^2}=1$) corresponding to the eigenvalue $\lambda_0$.
Using H\"older inequality we get
\begin{align*}
e^{\lambda_0\frac{t}{2}}\phi(x)=& T_2(t/2)\phi(x)\notag\\
=&\int_{\R^N}k_\mu(t/2,x,y)\phi(y)\,d\mu(y)\notag\\
\le&\left(\int_{\R^N}k_\mu(t/2,x,y)^2\,d\mu(y)\right)^{\frac{1}{2}}\notag\\
=& k_\mu(t,x,x)^{\frac{1}{2}},
\end{align*}
for any  $t>0$ and any $x\in \R^N$.
The assertion now follows from Proposition \ref{estim-eigenfunct}.
\end{proof}

We now state the main result of this section.

\begin{theorem}\label{kernel-estimates}
If $N>2,\,\alpha \in [0,2)$ and $\beta >2$, then
\begin{eqnarray*}
0<k(t,x,y)\le \frac{Ke^{\lambda_0 t}e^{ct^{-b}}f_0(x)f_0(y)}{1+|y|^\alpha},\qquad\;\,t>0,\;\,x,y\in\R^N\setminus B(0,1),
\end{eqnarray*}
where $K,\,c$ are positive constants, $b>\frac{\beta +2}{\beta -2}$ and
$f_0$ is given by \eqref{flambda}.
\end{theorem}

\begin{proof}
The case $\alpha =0$ is already known, see \cite[Cor. 4.5.5 and Cor. 4.5.8]{davies}. Let us consider the case when $\alpha \in (0,2)$.
We split the proof into two steps. In the first one we estimate the function $k_{\mu}(t,\cdot,\cdot)$ for $t\in (0,1]$ and prove that
\begin{equation}\label{intrins-ultra}
k_\mu(t,x,y)\le K_1e^{ct^{-b}}\phi(x)\phi(y),\qquad\;\, 0<t\le 1,\;\, x,y\in \R^N,
\end{equation}
for some positive constant $K_1$. In the second one we prove that
\begin{equation}
k_{\mu}(t,x,y)\le K_2e^{\lambda_0t}\phi(x)\phi(y),\qquad\;\,t>1,\;\,x,y\in\R^N.
\label{tge1}
\end{equation}
Combining \eqref{intrins-ultra} and \eqref{tge1}, taking \eqref{k-mu} and Corollary \ref{estim-eigenfunct0} into account,
the assertion follows at once.

{\em Step 1.} Estimate \eqref{intrins-ultra} can be proved adapting the arguments used in \cite[Subsect. 4.4 and Subsect. 4.5]{davies}. For this reason we do not
elaborate the proof but we just check the crucial points, which are the estimates
\begin{equation}\label{nash}
\int_{\R^N}g|f|^2\,d\mu \le b_0\|g\|_{L^{N/2}_\mu}a_\mu(f,f),\qquad\;\, f\in D(a_\mu),\;\,g\in L^{N/2}_\mu,
\end{equation}
for some positive constant $b_0$, independent of $f$ and $g$, and
\begin{equation}
\int_{\R^N}W|f|^2\,d\mu \le \varepsilon(a_\mu(f,f)+\lambda_0\|f\|_{L^2_\mu}^2)+(c\varepsilon^{-\frac{\gamma}{\beta -\gamma}}-\varepsilon \lambda_0)\|f\|^2_{L^2_\mu},\qquad\;\, f\in D(a_\mu),
\label{stima-W}
\end{equation}
for any $\varepsilon>0$, any $\gamma\in (\beta/2+1,\beta)$ and some positive constant $c$. Here,
$W(x)=|x|^{\gamma}$ for any $x\in\R^N$.

To prove \eqref{nash} it suffices to show that the semigroup $(e^{tA_{\mu}})_{t\ge 0}$ is ultracontractive and
\begin{equation}
\|e^{tA_\mu}f\|_\infty\le Ct^{-\frac{N}{4}}\|f\|_{L^2_\mu},\qquad\;\, t>0,\,f\in L^2_\mu,
\label{iper-contr}
\end{equation}
for some positive constant $C$, independent of $t$.
Theorem 2.4.2 in \cite{davies} will then imply that there exists a constant $b_0>0$ such that
\begin{eqnarray*}
\|f\|_{L^{2N/(N-2)}_{\mu}}^2\le b_0a_{\mu}(f,f),\qquad\;\,f\in D(a_{\mu}),
\end{eqnarray*}
and H\"older inequality will yield estimate \eqref{stima-W}.

So, let us prove \eqref{iper-contr}. For this purpose, we denote by $(S(t))_{t\ge 0}$ the analytic semigroup
generated by the realization of the operator $a\Delta$ in $L^2_{\mu}$.
The results in \cite[Thm. 2.14]{metafune-spina-3} show that each operator
$S(t)$ is a contraction in $L^{\infty}(\R^N)$,
$S(t)\in L(L^1_{\mu}(\R^N),L^{\infty}(\R^N))$ and $\|S(t)\|_{L(L^1_{\mu},L^{\infty}(\R^N))}\le C_1t^{-N/2}$ for some positive constant $C_1$, independent of $t$. Stein interpolation theorem implies that $S(t)$ is ultracontractive and $\|S(t)f\|_{\infty}\le C_1^{1/2}t^{-N/4}\|f\|_{L^2_{\mu}}$.
To complete the proof of \eqref{iper-contr} it now suffices to show that $S(t)f\le e^{tA_{\mu}}f$ for any $t>0$ and any nonnegative $f\in L^2_{\mu}$.
In fact, we prove such a property for any nonnegative $f\in C^{\infty}_c(\R^N)$.
By Proposition \ref{prop-cons-Linf} we know that both the functions $e^{tA_{\mu}}f=T_2(t)f$ and $S(\cdot)f$ belong to $C_b([0,\infty)\times\R^N)\cap C^{1,2}((0,\infty)\times\R^N)$ and their difference $v$ satisfies the differential inequality $D_tv-a\Delta v\le 0$ and vanishes at $t=0$.
By a variant of the classical maximum principle (see e.g., \cite[Thm. 4.1.3]{libro}), we can infer that $v\le 0$.
Hence, $e^{tA_{\mu}}f\le S(t)f$ for any $t\ge 0$.

Estimate \eqref{stima-W} follows at once observing that
\begin{eqnarray*}
W(x)\le \varepsilon (V(x) +\lambda_0)+c\varepsilon^{-\frac{\gamma}{\beta -\gamma}}-\varepsilon \lambda_0,\qquad\;\,x\in\R^N,
\end{eqnarray*}
where $c=(\beta^{-1}\gamma)^{\gamma/(\beta-\gamma)}-(\beta^{-1}\gamma)^{\beta/(\beta-\gamma)}$.

{\em Step 2.} To estimate the function $k_{\mu}(t,\cdot,\cdot)$ for $t>1$, we use the Chapman-Kolmogorov equation and the symmetry
of $k_{\mu}(t,\cdot,\cdot)$ to infer that
\begin{eqnarray*}
k_{\mu}(t,x,y)=\int_{\R^N}k_{\mu}(t-1/2,x,z)k_{\mu}(1/2,y,z)d\mu(z),\qquad\;\,t>1/2,\;\,x,y\in\R^N.
\end{eqnarray*}
By Step 1, the function $k_{\mu}(1/2,y,\cdot)$ belongs to $L^2_{\mu}$. Hence,
\begin{eqnarray*}
k_{\mu}(t,x,y)=(e^{(t-\frac{1}{2})A_{\mu}}k_{\mu}(1/2,y,\cdot))(x),\qquad\;\,t>1/2,\;\,x,y\in\R^N.
\end{eqnarray*}
From formula \eqref{kmu-0} and estimate \eqref{spectr-behav} we now deduce that
\begin{align*}
k_{\mu}(t,x,x)=&\int_{\R^N}|k_{\mu}(t/2,x,y)|^2d\mu(y)\\
\le & M_2e^{\lambda_0(t-1)}\|k_{\mu}(1/2,x,\cdot)\|_{L^2_{\mu}}^2\\
= & M_2e^{\lambda_0(t-1)}k_{\mu}(1,x,x)\\
\le & K_2e^{\lambda_0t}(\phi(x))^2,
\end{align*}
for any $t>0$, any $x\in\R^N$ and some positive constant $K_2$.
Using the inequality $k_{\mu}(t,x,y)\le (k_{\mu}(t,x,x))^{1/2}(k_{\mu}(t,y,y))^{1/2}$, which holds for any $t>0$
and any $x,y\in\R^N$ and follows from the Chapman-Kolmogorov equation, we get \eqref{tge1}.
\end{proof}

Let $\lambda_j$ be an eigenvalue of $A_2$ and denote by $\psi_j$ any normalized (i.e. $\|\psi_j\|_{L^2(\R^N)}=1$) eigenfunction associated to $\lambda_j$. Then,
by \eqref{kmu-0}, with $2t$ instead of $t$, we get
\begin{align*}
e^{\lambda_j t}|\psi_j(x)|=&\left|\int_{\R^N}k_\mu(t,x,y)\psi_j(y)\,d\mu(y)\right|\\
\le & \left(\int_{\R^N}k_\mu(t,x,y)^2d\mu(y)\right)^{\frac{1}{2}}\|\psi_j\|_{L^2_\mu}\\
= & (k_\mu(2t,x,x))^{\frac{1}{2}},
\end{align*}
for any $t>0$ and any $x\in\R^N$.
So, by Theorem \ref{kernel-estimates} we obtain
\begin{corollary}
If the assumptions of Theorem $\ref{kernel-estimates}$ hold, then all normalized eigenfunctions $\psi_j$ of $A_2$ satisfy
\begin{eqnarray*}
|\psi_j(x)|\le C|x|^{\frac{\alpha-\beta}{4}-\frac{N-1}{2}}\exp\left(-\int_1^{|x|}\frac{s^{\beta/2}}{(1+s^\alpha)^{1/2}}ds\right),
\end{eqnarray*}
for all $x\in\R^N\setminus B(0,1)$ and a constant $C>0$.
\end{corollary}

\subsection{A slightly more general class of elliptic operators}
Let us consider the operator $\mathcal{B}$, defined on smooth functions $u$ by
\begin{eqnarray*}
\mathcal{B}u=a\sum_{j,k=1}^ND_k(q_{kj}D_ju)-Wu,
\end{eqnarray*}
(where, as usual, $a(x)=(1+|x|^{\alpha})$ for any $x\in\R^N$) under the following set of assumptions:

\begin{hyp}
\label{hyp-1}\begin{enumerate}[\rm (i)]
\item
the coefficients $q_{kj}=q_{jk}$ belong to $C_b(\R^N)\cap W^{1,\infty}_{\rm loc}(\R^N)$ for any $j,k=1,\ldots,N$
and there exists a positive constant $\eta$ such that
\begin{eqnarray*}
\eta |\xi|^2  \le \sum_{j,k=1}^Nq_{kj}(x)\xi_k\xi_j,\qquad\;\,x,\xi \in \R^N;
\end{eqnarray*}
\item
$W\in L^1_{\rm loc}(\R^N)$ satisfies $W(x)\ge |x|^\beta$  for any $x\in \R^N$ and some $\beta>2$;
\item
$\alpha\in [0,2)$ and $D_j q_{kj}(x) = {\rm o}(|x|^{\frac{\beta-\alpha}{2}})$ as $|x| \to \infty$.
\end{enumerate}
\end{hyp}

On $L^2_\mu$ we define the bilinear form
\begin{align*}
b_\mu(u,v)=\sum_{j,k=1}^N\int_{\R^N}q_{kj}D_kuD_j\overline{v}\,dx+\int_{\R^N}Wu\overline{v}\,d\mu,\qquad\;\,u,v\in D(b_{\mu}),
\end{align*}
where $D(b_{\mu})=\{u\in L^2_{\mu}: W^{1/2}\in L^2_{\mu},\,\nabla u\in (L^2(\R^N))^N\}$.
Since $b_{\mu}$ is a symmetric, accretive and closable form, we can associate a positive strongly continuous semigroup in $L^2_{\zeta^2\mu}$ (we refer the reader
again to \cite[Prop. 1.51, Thms. 1.52, 2.6, 2.13]{ouhabaz}).
The same arguments as at the beginning of this section show
that the infinitesimal generator $B_{\mu}$ of this semigroup is the realization in $L^2_{\mu}$
of the operator $\mathcal{B}$ with domain $D(B_{\mu})=\{u\in D(b_{\mu})\cap W^{2,2}_{\rm loc}(\R^N): \mathcal{B}u\in L^2_{\mu}\}$.

In this subsection we prove upper estimates for the kernel $p_{\mu}$ of the semigroup $(e^{tB_{\mu}})_{t\ge 0}$.
For this purpose, for any $\theta>0$ we introduce the sesquilinear form $a_{\mu,\theta}$ defined by
\begin{eqnarray*}
a_{\mu,\theta}(u,v)=\int_{\R^N}\nabla u\cdot\nabla\overline{v}\,dx+\theta^2\int_{\R^N}Vu\overline{v}d\mu,\qquad\;\,u,v\in D(a_{\mu,\theta})=
D(a_{\mu}),
\end{eqnarray*}
where $V(x)=|x|^{\beta}$ for any $x\in\R^N$.
The arguments in the first part of this section can be applied to the
analytic semigroup associated to the form $a_{\mu,\theta}$ in $L^2_{\mu}$ and its kernel
$k_{\mu,\theta}$. In particular, arguing as in the proof of \eqref{intrins-ultra} and \eqref{tge1} one can show that
\begin{equation}
0<k_{\mu,\theta}(t,x,y)\le  K_{\theta}e^{\lambda_{0,\theta}t}e^{\tilde c_{\theta}t^{-b}}\phi_{0,\theta}(x)\phi_{0,\theta}(y),\qquad\;\,t>0,\;\,x,y\in\R^N,
\label{kernel-estim-theta}
\end{equation}
where $\tilde c_{\theta}$ and $K_{\theta}$ are positive constants, $\lambda_{0,\theta}$
is the largest (negative) eigenvalue of the minimal realization of
operator ${\mathcal A}_{\theta}$ in $L^2(\R^N)$, and $\phi_{0,\theta}$ is a corresponding
positive and bounded eigenfunction. Moreover, there exist two positive constants $C_{1,\theta}$ and $C_{2,\theta}$ such that
\begin{eqnarray*}
C_{1,\theta}\le |x|^{-\frac{\alpha-\beta}{4}+\frac{N-1}{2}}\exp\left (-\theta\int_1^{|x|}\frac{s^{\beta/2}}{(1+s^\alpha)^{1/2}}ds\right )\phi_{0,\theta}(x)\le C_{2,\theta},
\end{eqnarray*}
for any $x\in\R^N\setminus B(0,1)$.

In the proof of Theorem \ref{thm-ultimo} we will also need the precise asymptotic behavior of $|\nabla \phi_{0,\theta}|$.
\begin{proposition}\label{estim-gradient}
Assume that $\alpha\in [0,2)$ and $\beta>0$. Then,
\begin{eqnarray*}
\lim_{|x|\to \infty}\frac{|\nabla \phi_{0,\theta}(x)|^2}{(\phi_{0,\theta}(x))^2}\cdot \frac{1+|x|^\alpha}{|x|^\beta}=\frac{1}{\theta^2}.
\end{eqnarray*}
\end{proposition}
\begin{proof}
Since $\phi_{0,\theta}$ is radial (see Proposition \ref{thm-eigenfunct}), there exists a function $\phi_{\star,\theta}$ such that
$\phi_{0,\theta}(x)=\phi_{\star,\theta}(|x|)$ for any $x\in\R^N$. Let us consider the function $v$ defined by
$v(r)=|r|^{(N-1)/2}\phi_{\star,\theta}(-r)$ for any $r\ge 0$.
As it is easily seen,
\begin{eqnarray*}
v''(r)=v(r)\left(\frac{\theta^2|r|^\beta +\lambda_0}{1+|r|^\alpha}+\frac{N^2-4N+3}{4|r|^2}\right ),\qquad\;\,r<0.
\end{eqnarray*}
By \cite[Chpt. 6, Thm. 2.1]{olver} we know that there exist two solutions $w_1$ and $w_2$ of the previous equation, given by the following formula:
\begin{align*}
w_j(r)=&\left(\frac{1+|r|^\alpha}{\theta^2|r|^\beta +\lambda_0}\right)^{\frac{1}{4}}
\exp\bigg((-1)^j\int_{-1}^r\frac{(\theta^2|s|^\beta +\lambda_0)^{1/2}}{(1+|s|^\alpha)^{1/2}}\,ds\bigg)(1+\varepsilon_j(r)),
\end{align*}
for $j=1,2,$,
where $\varepsilon_j(r)$ and $(1+|r|^\alpha)^{1/2}(\theta^2|r|^\beta +\lambda_0)^{-1/2}\varepsilon_j'(r)$ tend
to $0$ as $r\to -\infty$. This last assertion follows from applying \cite[Formula (2.04)]{olver} noting
that the function $F$ in \cite[Formula (2.01)]{olver} has bounded variation in $(-\infty,a]$ for any $a<0$ since it is therein bounded and Lipschitz continuous.
It then follows that $w_1$ and $w_2$ are linearly independent since $w_1$ is unbounded whereas $w_2$ is bounded in $(-\infty,0)$. Hence, $v$ is a linear combination of the functions
$w_1$ and $w_2$.
Since $\phi_{\star,\theta}$ is bounded,
and $r\mapsto r^{-(N-1)/2}w_1(r)$ tends to $\infty$ as $r\to -\infty$, we obtain that $v=w_2$, i.e.,
\begin{eqnarray*}
\phi_{\star,\theta}(r)=cr^{-\frac{N-1}{2}}\left(\frac{1+r^\alpha}{\theta^2r^\beta +\lambda_0}\right)^{\frac{1}{4}}
\exp\bigg(-\int_1^r\frac{(\theta^2s^\beta +\lambda_0)^{1/2}}{(1+s^\alpha)^{1/2}}\,ds\bigg)(1+\varepsilon_1(r)),
\end{eqnarray*}
for any $r>0$ and some positive constant $c$.

Now, a direct computation yields
\begin{align*}
\phi_{\star,\theta}'(r)=\phi_{\star,\theta}(r) \bigg(&-\frac{N-1}{2r}+\frac{(\alpha -\beta)\theta^2r^{\alpha +\beta -1}+\alpha \lambda_0r^{\alpha -1}-\beta\theta^2 r^{\beta -1}}{4(1+r^\alpha)(\theta^2r^\beta +\lambda_0)}\\
&-\frac{(\theta^2r^\beta +\lambda_0)^{1/2}}{(1+r^\alpha)^{1/2}}+\frac{\varepsilon_1'(r)}{1+\varepsilon_1(r)}\bigg ),
\end{align*}
for any $r>0$.
Observing that the leading term (as $r\to \infty$) in the round brackets is the function $r\mapsto -(\theta^2r^{\beta}+\lambda_0)^{1/2}(1+r^{\alpha})^{-1/2}$, the assertion follows at once.
\end{proof}

\begin{theorem}
\label{thm-ultimo}
Assume that Hypotheses $\ref{hyp-1}$ are satisfied and
let
\begin{eqnarray*}
\Lambda:=\sup_{x,\xi\in\R^N\setminus\{0\}}|\xi|^{-2}\sum_{j,k=1}^Nq_{kj}(x)\xi_k\xi_j.
\end{eqnarray*}
Then, for any $\theta\in (0,\Lambda^{-1/2})$, we have
\begin{align}
p_\mu(t,x,y)\le &M_{\theta}e^{\lambda_{0,\theta}t}e^{c_{\theta}t^{-b}}|xy|^{\frac{\alpha-\beta}{4}-\frac{N-1}{2}}\notag\\
&\qquad\times\exp\left (-\theta\int_1^{|x|}\frac{s^{\beta/2}}{(1+s^\alpha)^{1/2}}ds-\theta\int_1^{|y|}\frac{s^{\beta/2}}{(1+s^\alpha)^{1/2}}ds\right ),
\label{estim-pmu}
\end{align}
for any $t>0$ and $x,y\in\R^N\setminus B(0,1)$, where $M_{\theta}$, $c_{\theta}$ are positive constants, and $b>\frac{\beta +2}{\beta -2}$.
\end{theorem}

\begin{proof}
The assertion can be proved adapting the arguments in \cite{ouhabaz-rhandi} to our situation. For the reader's convenience we
give the main ideas of the proof and some details. To overcome cumbersome notations, throughout the proof we do not stress the
dependence of the constants on $\theta$.

Let us denote by $\zeta$ any positive and smooth function such that
\begin{eqnarray*}
\zeta(x)=|x|^{\frac{\alpha-\beta}{4}-\frac{N-1}{2}}\exp\left (-\theta\int_1^{|x|}\frac{s^{\beta/2}}{(1+s^\alpha)^{1/2}}ds\right ),\qquad\;\,x\in
\R^N\setminus B(0,1).
\end{eqnarray*}
Since the ratio $\phi_{0,\theta}^{-1}\zeta$ is bounded from below and above by two positive constants, proving \eqref{estim-pmu}
is equivalent to showing that
\begin{equation}
\zeta^{-1}(x)p_{\mu}(t,x,y)\zeta^{-1}(y)\le M_1e^{ct^{-b}}e^{\lambda_{0,\theta}t},\qquad\;\,t>0,\;\,x,y\in\R^N,
\label{stima-nucleo-1}
\end{equation}
for some positive constants $M$ and $c$. Note that
the left-hand side of \eqref{stima-nucleo-1} is the kernel of the semigroup
$({\mathcal I}_{\zeta}^{-1}e^{tB_{\mu}}{\mathcal I}_{\zeta})_{t\ge 0}$ in $L^2_{\zeta^2\mu}$, where
${\mathcal I}_{\psi}:L^2_{\psi^2\mu}\to
L^2_{\mu}$ is the isometry defined by ${\mathcal I}_{\psi}f=f\psi$ for any $f\in L^2_{\zeta^2\mu}$ and any positive measurable function $\psi$.
Clearly, this semigroup (which from now on we denote by $(e^{t\tilde B_{\mu}})_{t\ge 0}$) is associated to
the form $\tilde b_{\mu}$ on $L^2_{\zeta^2\mu}$, defined by
$\tilde b_{\mu}(u,v)=b_{\mu}(\zeta u,\zeta v)$ for any $u,v\in D(\tilde b_{\mu})=\{u: \zeta u\in D(b_{\mu})\}$.

The main step of the proof consists in establishing \eqref{stima-nucleo-1} for $t\in (0,1]$. Indeed, once it is proved
for $t\in (0,1]$, \eqref{stima-nucleo-1} can be extended to any $t>0$ arguing as in the last part of the proof of Theorem
\ref{kernel-estimates}.

To establish \eqref{stima-nucleo-1} for $t\in (0,1]$, one has to prove the following facts:
\begin{enumerate}[{\rm (i)}]
\item
$u\wedge 1:=\min\{u,1\}\in D(\tilde a_{\mu,\theta})$ (resp. $D(\tilde b_{\mu})$) for any nonnegative $u\in D(\tilde a_{\mu,\theta})$ (resp.
$D(\tilde b_{\mu})$);
\item
the semigroup $(e^{t\tilde B_{\mu}})_{t\ge 0}$ and the semigroup
$(e^{t\tilde A_{\mu,\theta}})_{t\ge 0}$, associated to the form $\tilde a_{\mu,\theta}=a_{\mu,\theta}(\zeta\cdot,\zeta\cdot)$
with domain $D(\tilde a_{\mu,\theta})=D(\tilde b_{\mu})$, are positive, they map $L^{\infty}(\R^N)$ into itself and satisfy the estimates
\begin{eqnarray*}
\|e^{t\tilde A_{\mu,\theta}}\|_{L(L^{\infty}(\R^N))}\le e^{C_1t},\qquad\;\,
\|e^{t\tilde A_{\mu,\theta}}\|_{L(L^{\infty}(\R^N))}\le e^{C_1t},\;\,t>0,
\end{eqnarray*}
for some positive constant $C_1$;
\item
the Log-Sobolev inequality
\begin{align}
\;\;\;\;\;\int_{\R^N}u^2(\log u)\zeta^2d\mu\le\varepsilon\tilde b_{\mu}(u,u)
+\|u\|_{L^2_{\zeta^2\mu}}^2\log\|u\|_{L^2_{\zeta^2\mu}}+c(1+\varepsilon^{-b})\|u\|_{L^2_{\zeta^2\mu}}^2,
\label{LSI}
\end{align}
holds true for any nonnegative $u\in D(\tilde b_{\mu})\cap L^1_{\zeta^2\mu}\cap L^{\infty}(\R^N)$, where $c$ is the constant in
\eqref{stima-nucleo-1}.
\end{enumerate}
From \eqref{LSI} it follows immediately that
the form $\hat b_{\mu,\theta}(\cdot,\cdot):=\tilde b_{\mu,\theta}(\cdot,\cdot)+C_1(\cdot,\cdot)_{\zeta^2\mu}$,
satisfies the Log-Sobolev inequality
\begin{eqnarray*}
\int_{\R^N}u^2(\log u)\zeta^2d\mu\le\varepsilon\hat b_{\mu}(u,u)
+\|u\|_{L^2_{\zeta^2\mu}}^2\log\|u\|_{L^2_{\zeta^2\mu}}+c(1+\varepsilon^{-b})\|u\|_{L^2_{\zeta^2\mu}}^2,
\end{eqnarray*}
for any nonnegative $u\in D(\tilde b_{\mu})\cap L^1_{\zeta^2\mu}\cap L^{\infty}(\R^N)$,
and it is associated to the Markov semigroup $(e^{t(-C_1+\tilde B_{\mu})})_{t\ge 0}$.
Hence, combining \cite[Lemma 2.1.2, Cor. 2.2.8 and Ex. 2.3.4]{davies},
estimate \eqref{stima-nucleo-1} follows with $t\in (0,1]$.
\medskip

Let us elaborate with some details the previous three points.

Property (i) is satisfied with the set
$D:=\{u: u\phi_{0,\theta}\in D(a_{\mu,\theta})\}$ instead of
$D(\tilde a_{\mu,\theta})$, as it follows
from \cite[Cor. 2.17]{ouhabaz} applied to the form $a_{\mu,\theta}(\phi_{0,\theta}\cdot,\phi_{0,\theta}\cdot)$, which has
$D$ as a domain, and it is associated to the $L^{\infty}$-contractive semigroup $({\mathcal I}_{\phi_{0,\theta}}^{-1}e^{tA_{\mu,\theta}}{\mathcal I}_{\phi_{0,\theta}})_{t\ge 0}$.
On the other hand the spaces $D$ and $D(\tilde a_{\mu,\theta})$ coincide.
This follows from recalling that the ratio $\phi_{0,\theta}^{-1}\zeta$ is bounded from below and
above in $\R^N$ by two positive constants and from using Proposition \ref{estim-gradient} and a straightforward computation
to estimate
\begin{eqnarray*}
|\nabla (\zeta^{-1}\phi_{0,\theta})|
\le C_2(|\phi_{0,\theta}^{-1}\nabla\phi_{0,\theta}|+|\zeta^{-1}\nabla\zeta|)
\le C_3V^{1/2}\le C_3W^{1/2},
\end{eqnarray*}
for some positive constants $C_2$ and $C_3$. Hence, property (i) is satisfied by any nonnegative $u\in D(\tilde a_{\mu,\theta})$.
Since $D(\tilde b_{\mu})\subset D(\tilde a_{\mu,\theta})$ and $|W^{1/2}(u\wedge 1)|\le W^{1/2}|u|$ for
any $u\in D(\tilde b_{\mu})$, it is easy to check that property (ii) is satisfied also by any nonnegative
$u\in D(\tilde b_{\mu})$.

(ii) The positivity of the semigroups follows from \cite[Thm. 2.6, ``$4)\Rightarrow 1)$'']{ouhabaz}.
To prove that they map $L^{\infty}(\R^N)$  into itself, it suffices to observe that
the functions
\begin{align*}
V_a=&\theta^2V-\zeta^{-1}a\Delta\zeta,\qquad
W_b=&W-\zeta^{-1}a\sum_{i,j=1}^ND_iq_{ij}D_j\zeta-\zeta^{-1}a\sum_{i,j=1}^Nq_{ij}D_{ij}\zeta,
\end{align*}
are bounded from below by $-C_1$ and their moduli can be controlled by a constant times, respectively, the functions $V$ and $W$.
This and property (i) allow us to show that
\begin{align*}
\tilde a_{\mu,\theta}(u\wedge 1,(u-1)^+)=&\int_{\R^N}V_a(u\wedge 1)(u-1)^+\zeta^2d\mu\\
\ge &-C_1\int_{\R^N}(u\wedge 1)(u-1)^+\zeta^2d\mu,\\[1mm]
\tilde b_{\mu,\theta}(v\wedge 1,(v-1)^+)=&\int_{\R^N}W_b(v\wedge 1)(v-1)^+\zeta^2d\mu\\
\ge &-C_1
\int_{\R^N}(v\wedge 1)(v-1)^+\zeta^2d\mu,
\end{align*}
for any nonnegative $u\in D(\tilde a_{\mu,\theta})$ and $v\in D(\tilde b_{\mu})$.
Applying \cite[Cor. 2.17, ``$3)\Rightarrow 1)$'']{ouhabaz} to the forms
$\tilde a_{\mu,\theta}(\cdot,\cdot)+C_1(\cdot,\cdot)_{\zeta^2\mu}$ and
$\tilde b_{\mu}(\cdot,\cdot)+C_1(\cdot,\cdot)_{\zeta^2\mu}$,
and taking property (i) into account, (ii) follows.

(iii) Since $\tilde b_{\mu}(u,u)\ge \min\{\mu,\theta^{-1}\}\tilde a_{\mu,\theta}(u,u)$ for any $u\in D(\tilde b_{\mu})
\subset D(\tilde a_{\mu,\theta})$, it suffices to
prove the Log-Sobolev inequality for the form $\tilde a_{\mu}$.
For this purpose, we observe that the semigroup
$(e^{t\tilde A_{\mu,\theta}})_{t\ge 0}$ is ultracontractive and it satisfies
\begin{eqnarray*}
\|e^{tA_{\mu,\theta}}\|_{L(L^2_{\zeta^2\mu},L^{\infty}(\R^N))}\le C_4e^{\lambda_{0,\theta}t}e^{ct^{-b}},\qquad\;\,t>0.
\end{eqnarray*}
This property follows from  the kernel estimate
\eqref{kernel-estim-theta}, which shows that the kernel $\tilde k_{\mu,\theta}$ associated to the semigroup $(e^{t\tilde A_{\mu,\theta}})_{t\ge 0}$ satisfies the estimate
\begin{eqnarray*}
\tilde k_{\mu,\theta}(t,x,y)=\zeta^{-1}(x)k_{\mu,\theta}(t,x,y)\zeta^{-1}(y)=\tilde k_{\mu,\theta}(t,x,y)\le C_4e^{\lambda_{0,\theta}t+ct^{-b}},
\end{eqnarray*}
for any $t>0$, any $x,y\in\R^N$ and some positive constant $C_4$, and the fact that $L^2_{\zeta^2\mu}$ is continuously embedded into $L^1_{\zeta^2\mu}$. We can thus apply \cite[Thm. 2.2.3]{davies} (note that its proof works as well also in the case when
the semigroup is not $L^{\infty}$-contractive but its $L^{\infty}$-norm is bounded on bounded sets of $[0,+\infty)$) obtaining \eqref{LSI} with $\tilde a_{\mu,\theta}$ replaced by $\tilde b_{\mu}$, and with $c$ being replaced by a constant $\hat c$.
\end{proof}

\appendix
\section{ }

This appendix contains all technical results that we need to prove Proposition \ref{prop-1} and Theorem \ref{prop-2}.
\begin{proposition}
\label{lemma-1}
Let $q:\R^N\to\R$ be a positive and continuous function such that $q(x)\le C|x|^2$ for any $x\in\R^N$ and some positive
constant $C$. Further, let $W\in C(\R^N)$ be a nonnegative function.
Then, for any $p\in (1,\infty)$, $C^{\infty}_c(\R^N)$ is dense in the space ${\mathcal Z}:=\{u\in W^{2,p}(\R^N):
q^{1/2}|\nabla u|, q|D^2u|, Wu\in L^p(\R^N)\}$, endowed with the norm
\begin{align*}
\|u\|_{\mathcal Z}=\|u\|_{L^p(\R^N)}+\|Wu\|_{L^p(\R^N)}+\|q^{1/2}|\nabla u|\,\|_{L^p(\R^N)}+\|q|D^2u|\,\|_{L^p(\R^N)}.
\end{align*}
\end{proposition}

\begin{proof}
Even if the proof can be obtained employing a standard argument, for the reader's convenience we enter details.
As a first step, take $u\in W^{2,p}(\R^N)$, with compact support, and regularize
it by convolution with standard mollifiers, obtaining a sequence $(u_n)\subset C^{\infty}_c(\R^N)$
converging to $u$ in $W^{2,p}(\R^N)$. Since ${\rm supp}(u_n)\subset {\rm supp}(u)+
B_1$, for any $n\in\N$, $q^{1/2}D_iu_n$, $qD^2_{ij}u_n$ and $Wu_n$ converge, respectively, to
$q^{1/2}D_iu$, $qD^2_{ij}u$ and $Wu$ in $L^p(\R^N)$, as $n$ tends to $\infty$, for any
$i,j=1,\ldots,N$. Hence, $u_n$ tends to $u$
in ${\mathcal Z}$. To conclude the proof, it suffices to show that
any function $u\in {\mathcal Z}$ can be approximated in the ${\mathcal Z}$-norm by
a sequence of compactly supported functions in $W^{2,p}(\R^N)$. For this purpose, to any fixed $u\in {\mathcal Z}$ we
associate the sequence $(u_n)$ defined as follows: $u_n=u\vartheta_n$ for any $n\in\N$,
where $\vartheta_n(x)=\vartheta(n^{-1}x)$ for any $x\in\R^N$ and $\vartheta$ is a smooth function such that
$\chi_{B_1}\le\vartheta\le\chi_{B_2}$.
By dominated convergence, $u_n$ and $Wu_n$ tend to $u$ and $Wu$ in $L^p(\R^N)$, respectively. Similarly, one has
\begin{align*}
\|q^{\frac{1}{2}}|\nabla u_n-\nabla u|\,\|_{L^p(\R^N)}^p
\le & 2^{p-1}\int_{\R^N}q^{\frac{p}{2}}|\nabla u|^p|\vartheta_n-1|^pdx\\[1mm]
&+\frac{2^{p-1}}{n^p}\sup_{n\le |x|\le 2n}q(x)^{\frac{p}{2}}\int_{\R^N}|\nabla\vartheta(n^{-1}\cdot)|^p
|u|^pdx\\[1mm]
\le& 2^{p-1}\int_{\R^N}q^{\frac{p}{2}}|\nabla u|^p|\vartheta_n-1|^pdx\\[1mm]
&+2^{2p-1}C^{\frac{p}{2}}\int_{\R^N}|\nabla\vartheta(n^{-1}\cdot)|^p|u|^pdx.
\end{align*}
Hence, from the dominated convergence theorem, we deduce that
the last side of the previous chain of inequalities vanishes as $n$ tends to $\infty$.

A completely similar computation shows that $qD_{ij}u_n$ tends to $qD_{ij}u$ in $L^p(\R^N)$ as $n\to\infty$, for any
$i,j=1,\ldots,N$. This completes the proof.
\end{proof}

The following interior $L^p$-estimates are crucial to prove that $A_p$ is a sectorial operator in the case when $\alpha\in [0,1)$.

\begin{proposition}
\label{prop-stime-interne}
Fix $p\in (1,\infty)$ and let $L$ be a uniformly second-order elliptic operator in non-divergence form with bounded and continuous coefficients. Further assume that the diffusion coefficients are $\gamma$-H\"older continuous for some $\gamma\in (0,1)$. Then, there exists a positive constant $\omega$ such that, for any $r>0$ and any $\lambda\in\C$ with ${\rm Re}\lambda\ge\omega$,
\begin{equation}
\|\,|D^2u|\,\|_{L^p(B_r)}\le M_r\left (\|\lambda u-Lu\|_{L^p(B_{2r})}+\|u\|_{L^p(B_{2r})}\right ),
\label{stima-interna-hessiano}
\end{equation}
for any $u\in W^{2,p}(\R^N)$ and  some positive constant $M_r$,
which depends on $r$, the sup-norm of the coefficients of the operator $L$, the ellipticity constant, the modulus of continuity of the diffusion coefficients of $L$, but it is independent of $u$ and $\lambda$.
\end{proposition}

\begin{proof}
It is well-known that there exist two positive constants $\omega$ and $C$ such that
\begin{equation}
\|\,|D^2v|\,\|_{L^p(\R^N)}\le C\|\lambda v-Lv\|_{L^p(\R^N)},
\label{stima-globale-hessiano}
\end{equation}
for any $v\in W^{2,p}(\R^N)$ and any $\lambda\in\C$ with ${\rm Re}\lambda\ge\omega$ (see e.g., \cite[Thm. 3.1.3]{lunardi}).

Let us fix $r>0$ and, for any $n\in\N$, set $r_n=2r(1-2^{-n-1})$. Clearly, $r_0=r$ and $\lim_{n\to\infty}r_n=2r$.
Let $(\varphi_n)$ be a sequence of smooth functions such that $\chi_{B_{r_n}}\le\varphi_n\le\chi_{B_{r_{n+1}}}$, $|\nabla\varphi_n|\le 2^ncr^{-1}$,
$|D^2\varphi_n|\le 4^ncr^{-2}$ in $\R^N$, for any $n\in\N$ and some positive constant $c$.
Fix $u\in W^{2,p}(\R^N)$ and apply estimate \eqref{stima-globale-hessiano} to the function $v_n=u\varphi_n$. We get
\begin{equation}
\|\,|D^2v_n|\,\|_{L^p(\R^N)}
\le C_r\left (\|\lambda u-Lu\|_{L^p(B_{2r})}+4^n\|u\|_{L^p(B_{2r})}
+2^n\|\,|\nabla v_{n+1}|\,\|_{L^p(\R^N)}\right ).
\label{stima-interna-1}
\end{equation}
The constant appearing in \eqref{stima-interna-1}, as well as in all the forthcoming estimates, are
all independent of $\lambda$, $u$ and $n$.
Using the interpolation inequality \eqref{interp}, and recalling that $\|v_{n+1}\|_{L^p(\R^N)}\le \|u\|_{L^p(B_{2r})}$, we can estimate
\begin{eqnarray*}
\|\,|\nabla v_{n+1}|\,\|_{L^p(\R^N)}\le c_N\varepsilon\|\,|D^2v_{n+1}|\,\|_{L^p(\R^N)}+\frac{c_N}{4\varepsilon}\|u\|_{L^p(B_{2r})},
\end{eqnarray*}
for any $\varepsilon>0$. Plugging this inequality into \eqref{stima-interna-1} and choosing
$\varepsilon=2^{-n-4}(C_rc_N)^{-1}$, we get
\begin{eqnarray*}
\|\,|D^2v_n|\,\|_{L^p(\R^N)}\le\frac{1}{16}\|\,|D^2v_{n+1}|\,\|_{L^p(\R^N)}
+C_r'\left (\|\lambda u-Lu\|_{L^p(B_{2r})}+4^n\|u\|_{L^p(B_{2r})}\right ).
\end{eqnarray*}
Let us multiply both the sides of the previous inequality by $16^{-n}$ and then sum over $n=0,\ldots,M$.
We obtain
\begin{align*}
&\|\,|D^2v_0|\,\|_{L^p(\R^N)}-\frac{1}{16^{M+1}}\|\,|D^2v_{M+1}|\,\|_{L^p(\R^N)}\\
\le & C_r''\|\lambda u-Lu\|_{L^p(B_{2r})}+C_r''\|u\|_{L^p(B_{2r})}.
\end{align*}
The assumptions on $\varphi_n$ imply that
$\|\,|D^2v_{M+1}|\,\|_{L^p(\R^N)}\le C4^{M+1}\|u\|_{W^{2,p}(\R^N)}$ for some positive constant $C$, independent of $M$.
Hence, we obtain \eqref{stima-interna-hessiano} letting $M\to\infty$, recalling that $v_0=u$ in $B_r$.
\end{proof}

\end{document}